\documentclass{article}
\usepackage{amsthm}

\usepackage[utf8]{inputenc}
\usepackage[english]{babel}
\usepackage{gensymb}
\usepackage{amsmath,amsfonts,amssymb}
\usepackage{subcaption} 
\usepackage{enumitem}
\usepackage{xfrac}
\usepackage{array}
\usepackage{esvect}
\usepackage{tikz}
\usetikzlibrary{shapes}
\usetikzlibrary{arrows}
\usepackage{textcomp}
\usepackage{algorithm}
\usepackage{algpseudocode}
\usepackage{cases}
\usepackage{todonotes}
\usepackage{hyperref}
\usepackage{dirtytalk}
\usepackage{graphicx,epstopdf}
\usepackage{xcolor}
\definecolor{coldeg1}{HTML}{a3be8c}
\definecolor{coldeg2}{HTML}{ebcb8b}
\definecolor{coldeg3}{HTML}{88c0d0}
\definecolor{coldeg4}{HTML}{d08770}
\definecolor{coldeg5}{HTML}{b48ead}
\definecolor{coldeg6}{HTML}{BF616A}

\title{Saturated Fully Leafed Tree-Like Polyforms and Polycubes}

\author{Alexandre Blondin Mass\'e
  \and Julien de Carufel
  \and Alain Goupil
}

\newcommand\N{\mathbb{N}}
\newcommand\Z{\mathbb{Z}}
\newcommand\R{\mathbb{R}}
\newcommand\dist{\mathrm{dist}}
\newcommand\depth{\emph{depth}}
\newcommand\grille{\mathcal{G}_2}
\newcommand\cube{\mathcal{G}_3}
\newcommand\grounded{\mathcal{GP}}
\newcommand\treelikepolyset{\mathcal{T_\squ}}
\newcommand\satpolyset{\mathcal{ST_\squ}}

\newcommand\satcubset{\mathcal{ST_\cub}}
\newcommand\ftcubset{4\mathcal{T}_i}

\newcommand\Ext{\mathrm{Ext}}
\newcommand\Int{\mathrm{Int}}
\newcommand\Hull{\mathrm{Hull}}
\newcommand\hex{\mathrm{hex}}
\newcommand\Hex{\mathrm{Hex}}
\newcommand\tri{\mathrm{tri}}
\newcommand\Tri{\mathrm{Tri}}
\newcommand\squ{\mathrm{squ}}
\newcommand\cub{\mathrm{cub}}

\def\graft{\; \triangleleft \;}
\newenvironment{enum}{\begin{enumerate}[label={\rm(\roman*)}, noitemsep, nolistsep]}{\end{enumerate}}

\newtheorem{theorem}{Theorem}[section]
\newtheorem{lemma}[theorem]{Lemma}
\newtheorem{proposition}[theorem]{Proposition}
\newtheorem{definition}[theorem]{Definition}
\newtheorem{corollary}[theorem]{Corollary}
\begin{document}

\maketitle

\begin{abstract}
  We present recursive formulas giving the maximal number of leaves in tree-like polyforms living in two-dimensional regular lattices and in tree-like polycubes   in the three-dimensional cubic lattice.
  We call these tree-like polyforms and polycubes \emph{fully leafed}.
  The proof relies on a combinatorial algorithm that enumerates rooted directed trees that we call abundant.
  In the last part, we concentrate on the particular case of polyforms and polycubes, that we call \emph{saturated}, which is the family of fully leafed structures that maximize the ratio  
  $\mbox{(number of leaves)}/\mbox{ (number of cells)}$.
  In the polyomino case, we present a bijection between the set of saturated tree-like polyominoes of size $4k+1$ and the set of tree-like polyominoes of size $k$.  
  We exhibit a similar bijection between the set of saturated tree-like polycubes of size $41k+28$ and a family of polycubes, called $4$-trees, of size $3k+2$.
\end{abstract}

\section{Introduction}
\label{S:intro}

Polyominoes and, to a lesser extent, polyhexes, polyiamonds and polycubes have been the object of important investigations in the past $30$ years either from a game theoretic or from a combinatorial point of view (see \cite{hr,Gu} and references therein).
Recall that a polyomino is an edge-connected set of unit cells in the square lattice that is invariant under translation.
There are two other regular lattices in the euclidian plane namely the hexagonal lattice and the triangular lattice which contain analogs of polyominoes respectively called polyhexes and polyiamonds.  
All these connected sets of planar cells are known under the name polyform. 
The 3D equivalent of a polyomino is called a polycube.
It is a face-connected set of unit cells in the cubic lattice, up to translation.

A central problem has been the search for the number of polyforms with $n$ cells where $n$ is called the size of the polyform.
This problem, still open, has been investigated from several points of view; asymptotic evaluation \cite{kr}, computer generation and counting \cite{je,kn,Re}, random generation \cite{hlm} and combinatorial description \cite{bfr,gpd,hr}.
Combinatorists have also concentrated their efforts in the description of various families of polyominoes and polycubes, such as convex polyominoes \cite{bmg}, parallelogram polyominoes \cite{aadhl,dv}, tree-like polyominoes \cite{gcn} and other families \cite{bmr,cfmrr,cdcj}. 
 
In this paper, we are interested in several related sets of polyforms: two-dimensional {\it tree-like polyforms} and three-dimensional {\it tree-like polycubes} which are acyclic in the graph theoretic sense.
Our main results are recursive expressions giving the maximal number of leaves of tree-like polyforms in the square, hexagonal and triangular regular lattices and also of tree-like polycubes of size $n$.
A tree-like polyform of size $n$ is called {\it fully leafed} when it contains the maximum number of leaves among all tree-like polyforms of size $n$.
The function $L_f(n)$ which gives the number of leaves in a fully leafed two-dimensional tree-like polyform with $n$ cells in the regular lattice $f$ is called the \emph{leaf function} of $f$.
Simillarly we denote by $L_\cub(n)$ the leaf function of the cubic lattice.
 
We also present explicit expressions for the number of saturated tree-like polyhexes and polyiamonds of given size $n$.
The structure of tree-like polyforms under investigation is similar to that which solves the maximum leaf spanning tree problem in grid graphs, one of the classical NP-complete problems described by Garey and Johnson in their seminal paper \cite{gj, bcglnv1}.
Both problems are concerned with the maximization of the number of leaves in subtrees, but they present a fundamental difference.
On one hand, spanning trees of a graph $G$ must contain all vertices of $G$.
On the other hand, induced subtrees $T$ of $G$ must contain every edge of $G$ between two vertices of $T$.
To our knowledge, these new classes of polyforms which are induced subgraphs of infinite regular lattice graphs present remarkable structure and properties that have neither been considered nor investigated yet. 
For example, the snake in the box problem \cite{ka}, which searches induced subtrees of maximal size with two leaves deserve, from our point of view, as much attention as the Hamiltonian path problem which search spanning trees with two leaves. 

The problem of finding the {\it maximum} number of {\it leaves} in tree-like polyforms extends naturally to the more general {\it Maximum Leaves in Induced Subtrees} (MLIS) problem, which consists in looking for induced subtrees having a maximum number of leaves in any simple graph.
Preliminary results about the MLIS problem can be found in \cite{bcglnv1} and \cite{bcglnv2}. 

This document is organized as follows.
In Section \ref{S:prelem} we introduce the concepts on graph theory and polyforms necessary for our treatment.
In Section \ref{S:polyominoes} we study fully leafed polyominoes and we introduce a general methodology for our proofs.
Section \ref{S:polyhexes-polyiamonds} focuses on the case of polyhexes and polyiamonds.
The more intricate case of tree-like polycubes is discussed in Section \ref{S:polycubes}. In particular, our proofs rely on an operation called \emph{graft union}, which allows to track efficiently the number of leaves.

In Section \ref{S:Saturated} we shift our attention to the family of saturated tree-like polyforms and polycubes and establish bijections for polyominos and polycubes that provide key informations for their enumeration.
Finally in Section \ref{S:conclusion} we conclude with asymptotic lower and upper bounds for the numbers $L_d(n)$ of leaves of $d$-dimensional tree-like polycubes and we sketch some directions for future work.

This manuscrit is an extended version of a paper presented at the 28th International Workshop on Combinatorial Algorithms (IWOCA 2017), held in Newcastle, Australia \cite{iwoca}.

\section{Preliminaries}\label{S:preliminaries}
\label{S:prelem}

Let $G = (V,E)$ be a simple graph, $u \in V$ and $U \subseteq V$.
The set of neighbors of $u$ in $G$ is denoted $N_G(u)$ and it is naturally extended to $U$ by defining $N_G(U) = \{u' \in N_G(u) \mid u \in U\}$.
For any subset $U \subseteq V$, the \emph{subgraph induced by $U$} is the graph $G[U] = (U, E \cap \mathcal{P}_2(U))$, where $\mathcal{P}_2(U)$ is the set of $2$-elements subsets of $V$.
The \emph{extension} of $G[U]$ is defined by $\Ext(G[U]) = G[U \cup N_G(U)]$ and the \emph{interior} of $G[U]$ is defined by $\Int(G[U]) = G[\Int(U)]$, where $\Int(U) = \{u' \in U \mid N_G(u') \subseteq U\}$.
Finally, the \emph{hull} of $G[U]$ is defined by $\Hull(G[U]) = \Int(\Ext(G[U]))$.

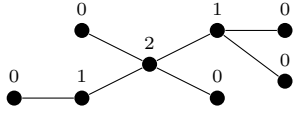
\begin{figure}[t]
  \centering
  \begin{tikzpicture}[scale=0.45, vertex/.style={fill=black, circle, minimum size=2mm, inner sep=0pt}]
  {\scriptsize
  \node[vertex, label=90:$2$] (a) at (0,0)   {};
  \node[vertex, label=90:$1$] (b) at (2,1)   {};
  \node[vertex, label=90:$0$] (c) at (2,-1)  {};
  \node[vertex, label=90:$0$] (d) at (4,-0.5){};
  \node[vertex, label=90:$0$] (e) at (4,1)   {};
  \node[vertex, label=90:$0$] (f) at (-2,1)  {};
  \node[vertex, label=90:$1$] (g) at (-2,-1) {};
  \node[vertex, label=90:$0$] (h) at (-4,-1) {};
  }
  \path (a) edge (b) edge (c) edge (f) edge (g);
  \path (b) edge (d) edge (e);
  \path (g) edge (h);
\end{tikzpicture}
  \caption{The depth of the vertices in a tree.}
  \label{F:depth}
\end{figure}

The \emph{square lattice} is the infinite simple graph $\grille = (\Z^2, A_4)$, where $A_4$ is the \emph{$4$-adjacency relation} defined by $A_4 = \{(p,p') \in \Z^2 \mid \dist(p, p') = 1\}$ and $\dist$ is the Euclidean distance of $\R^2$.
For any $p \in \Z^2$, the set $c(p) = \{ p' \in \R^2 \mid \dist_\infty(p, p') \leq 1/2\}$, where $\dist_\infty$ is the uniform distance of $\R^2$, is called the \emph{square cell} centered in $p$. The function $c$ is naturally extended to subsets of $\Z^2$ and subgraphs of $\grille$.
For any finite subset $U \subseteq \Z^2$, we say that $\grille[U]$ is a \emph{grounded polyomino} if it is connected.
The set of all grounded polyominoes is denoted by $\grounded$.
Given two grounded polyominoes $P = \grille[U]$ and $P' = \grille[U']$, we write $P \equiv_t P'$ (resp. $P \equiv_i P'$) if there exists a translation $T : \Z^2 \rightarrow \Z^2$ (resp. an isometry $I$ on $\Z^2$) such that $U' = T(U)$ (resp. $U' = I(U)$).
A \emph{fixed polyomino} (resp. \emph{free polyomino}) is then an element of $\grounded / \equiv_t$ (resp. $\grounded / \equiv_i$).
Clearly, any connected induced subgraph of $\grille$ corresponds to exactly one connected set of square cells via the function $c$.
Consequently, from now on, polyominoes will be considered as simple graphs rather than sets of edge-connected square cells. 

All definitions in the above paragraph are extended to the \emph{hexagonal lattice} with the \emph{$6$-adjacency relation}, the \emph{triangular lattice} with the \emph{$3$-adjacency relation} and the \emph{cubic lattice} with the \emph{$6$-adjacency relation}.
We thus extend the definition of \emph{cell}, \emph{grounded polyomino}, \emph{fixed polyomino} and \emph{free polyomino} to these regular lattices accordingly.

Grounded polyominoes and polycubes are connected subgraphs of $\grille$ and $\cube$ and the terminology of graph theory becomes available.
A (grounded, fixed or free) \emph{tree-like polyomino} is therefore a (grounded, fixed or free) polyomino whose associated graph is a tree.
\emph{Tree-like polyforms and polycubes} are defined similarly.
Observe that if $u,v$ are adjacent cells in a tree-like polyomino $T$ then $\deg_T(u)+\deg_T(v)\leq 6$.
This observation extends to polycubes in $\Z^d$ with $d \geq 2$ where we have $\deg_T(u) + \deg_T(v) \leq 2d + 2$.
In the figures, the vertices of graphs are colored according to their degree using the color palette below.
\begin{center}
  \begin{tabular}{|p{1.0cm}|cccccc|}
    \hline
    Degree & 1 & 2 & 3 & 4 & 5 & 6 \\\hline &&&&&& \\[-7pt]
    Color  & 
    \tikz[scale=0.5, baseline=1.2mm] \fill[coldeg1] (0,0) rectangle (1,1); &
    \tikz[scale=0.5, baseline=1.2mm] \fill[coldeg2] (0,0) rectangle (1,1); &
    \tikz[scale=0.5, baseline=1.2mm] \fill[coldeg3] (0,0) rectangle (1,1); &
    \tikz[scale=0.5, baseline=1.2mm] \fill[coldeg4] (0,0) rectangle (1,1); &
    \tikz[scale=0.5, baseline=1.2mm] \fill[coldeg5] (0,0) rectangle (1,1); &
    \tikz[scale=0.5, baseline=1.2mm] \fill[coldeg6] (0,0) rectangle (1,1); \\[4pt]\hline
  \end{tabular}
\end{center}

Let $T = (V,E)$ be any finite simple non empty tree.
We say that $u\in V$ is a \emph{leaf} of $T$ when $\deg_T(u) = 1$.
Otherwise $u$ is called an \emph{inner vertex} of $T$.
For any $d \in \N$, the number of vertices of degree $d$ is denoted by $n_d(T)$ and $n(T) = |V|$ is the number of vertices of $T$ which is also called the \emph{size} of $T$.
The \emph{depth} of $u\in V$ in $T$, denoted by $\depth_T(u)$, is defined recursively by
$$\depth_T(u) = \begin{cases}
  0, & \mbox{if $\deg_T(u) \leq 1$;} \\
  1 + \depth_{T'}(u), & \mbox{otherwise,}
\end{cases}$$
where $T'$ is the tree obtained from $T$ by removing all its leaves (see Figure~\ref{F:depth}).
Let $C$ be a tree whose set of inner vertices is $I$. We say that $C$ is a \emph{caterpillar} if $C[I]$ is a chain graph.

\section{Fully Leafed Tree-Like Polyominoes}\label{S:polyominoes}

In this section, we describe the number of leaves of fully leafed tree-like polyominoes.
For any integer $n \geq 2$, let the function $\ell_\squ(n)$ be defined as follows:
\begin{equation}\label{defl2}
  \ell_\squ(n) = \begin{cases}
    2,                 & \mbox{if $n = 2$;} \\
    n - 1,             & \mbox{if $n = 3,4,5$;} \\
    \ell_\squ(n - 4) + 2, & \mbox{if $n \geq 6$.}
  \end{cases}
\end{equation}

We claim that $\ell_\squ(n)=L_\squ(n)$ is the maximal number of leaves of a tree-like polyomino of size $n$. The first step is straightforward.
\begin{lemma}\label{L:family-2d}
For all $n \geq 2$, $L_\squ(n) \geq \ell_\squ(n)$.
\end{lemma}

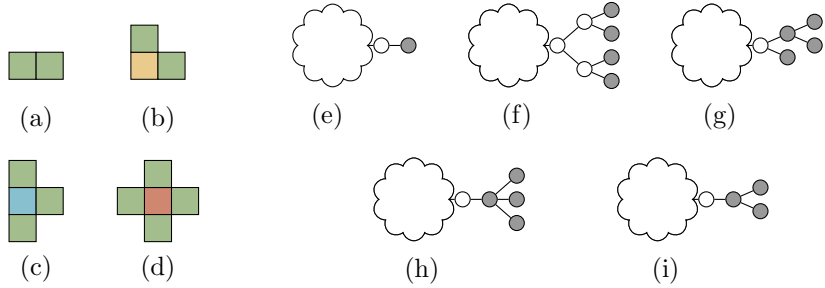
\begin{figure}[t]
  \centering
    \begin{tikzpicture}[scale=0.90]

  \begin{scope}[scale=.4, polyo/.style={draw=black, fill=black!30}]
    \begin{scope}
      \filldraw[polyo, fill=coldeg1] (0,0) rectangle (1,1);
      \filldraw[polyo, fill=coldeg1] (1,0) rectangle (2,1);
      \node at (1,-1.5) {(a)};
    \end{scope}
    \begin{scope}[xshift=4.5cm]
      \filldraw[polyo, fill=coldeg2] (0,0) rectangle (1,1);
      \filldraw[polyo, fill=coldeg1] (1,0) rectangle (2,1);
      \filldraw[polyo, fill=coldeg1] (0,1) rectangle (1,2);
      \node at (1,-1.5) {(b)};
    \end{scope}
    \begin{scope}[yshift=-5cm]
      \filldraw[polyo, fill=coldeg3] (0,0)  rectangle (1,1);
      \filldraw[polyo, fill=coldeg1] (1,0)  rectangle (2,1);
      \filldraw[polyo, fill=coldeg1] (0,1)  rectangle (1,2);
      \filldraw[polyo, fill=coldeg1] (0,-1) rectangle (1,0);
      \node at (1,-2.0) {(c)};
    \end{scope}
    \begin{scope}[xshift=5.0cm, yshift=-5cm]
      \filldraw[polyo, fill=coldeg4] (0,0)  rectangle (1,1);
      \filldraw[polyo, fill=coldeg1] (1,0)  rectangle (2,1);
      \filldraw[polyo, fill=coldeg1] (0,1)  rectangle (1,2);
      \filldraw[polyo, fill=coldeg1] (0,-1) rectangle (1,0);
      \filldraw[polyo, fill=coldeg1] (-1,0) rectangle (0,1);
      \node at (0.5,-2.0) {(d)};
    \end{scope}
  \end{scope}

  \begin{scope}[xshift=4.7cm, yshift=5mm, yscale=.35, xscale=.4, sommet/.style={draw=black, fill=white, minimum size=2mm, circle, inner sep=0pt}, prune/.style={sommet, fill=black!40}, every edge/.style={draw}, reste/.style={shape=cloud, draw=black, fill=white, minimum width=1.1cm, minimum height=1.1cm}]
    \begin{scope}
    \node[sommet] (a) at (2,0) {};
    \node[prune]  (b) at (3,0) {};
    \node[reste]  (r) at (0.2,0) {};
    \path (r) edge (a);
    \path (a) edge (b);
    \node at (0,-3) {(e)};
    \end{scope}
    \begin{scope}[xshift=6.5cm]
    \node[sommet] (a) at (1,0)    {};
    \node[sommet] (b) at (2,0)    {};
    \node[sommet] (c) at (3,1)    {};
    \node[sommet] (d) at (3,-1)   {};
    \node[prune]  (e) at (4,1.5)  {};
    \node[prune]  (f) at (4,0.5)  {};
    \node[prune]  (g) at (4,-0.5) {};
    \node[prune]  (h) at (4,-1.5) {};
    \node[reste]  (r) at (0.2,0)    {};
    \path (r) edge (a);
    \path (a) edge (b);
    \path (b) edge (c) edge (d);
    \path (c) edge (e) edge (f);
    \path (d) edge (g) edge (h);
    \node[reste]      at (0.2,0)    {};
    \node at (0.5,-3) {(f)};
    \end{scope}
    \begin{scope}[xshift=14cm]
    \node[sommet] (a) at (1,0)    {};
    \node[sommet] (b) at (2,0)    {};
    \node[prune]  (c) at (3,0.5)  {};
    \node[prune]  (d) at (3,-0.5) {};
    \node[prune]  (e) at (4,1)    {};
    \node[prune]  (f) at (4,0)    {};
    \node[reste]  (r) at (0.2,0)    {};
    \path (r.east) edge (a);
    \path (a) edge (b);
    \path (b) edge (c) edge (d);
    \path (c) edge (e) edge (f);
    \node[reste]      at (0.2,0)    {};
    \node at (0.5,-3) {(g)};
    \end{scope}
    \begin{scope}[yshift=-6.5cm, xshift=3cm]
    \node[sommet] (a) at (1,0)  {};
    \node[sommet] (b) at (2,0)  {};
    \node[prune]  (c) at (3,0)  {};
    \node[prune]  (d) at (4,0)  {};
    \node[prune]  (e) at (4,1)  {};
    \node[prune]  (f) at (4,-1) {};
    \node[reste]  (r) at (0.2,0)  {};
    \path (r.east) edge (a);
    \path (a) edge (b);
    \path (b) edge (c);
    \path (c) edge (d) edge (e) edge (f);
    \node[reste]  (r) at (0.2,0)  {};
    \node at (0.5,-3) {(h)};
    \end{scope}
    \begin{scope}[yshift=-6.5cm, xshift=12cm]
    \node[sommet] (a) at (1,0)    {};
    \node[sommet] (b) at (2,0)    {};
    \node[prune]  (c) at (3,0)    {};
    \node[prune]  (d) at (4,0.5)  {};
    \node[prune]  (e) at (4,-0.5) {};
    \node[reste]  (r) at (0.2,0)    {};
    \path (r.east) edge (a);
    \path (a) edge (b);
    \path (b) edge (c);
    \path (c) edge (d) edge (e);
    \node[reste]  (r) at (0.2,0)    {};
    \node at (0.5,-3) {(i)};
    \end{scope}
  \end{scope}
  \end{tikzpicture}
  \caption{Fully leafed tree-like polyominoes of size (a) $2$, (b) $3$, (c) $4$ and (d) $5$. The images (e), (f), (g), (h) and (i) depict the five cases of Lemma~\ref{L:upper-2d} (gray cells are removed).}
  \label{F:2d-case}
\end{figure}

\begin{proof}
  We build a family of tree-like polyominoes $\{T_n \mid n \geq 2\}$ whose number of leaves is given by \eqref{defl2}.
  For $n = 2,3,4,5$, the polyominoes $T_n$ respectively in (a), (b), (c) and (d) of Figure~\ref{F:2d-case} satisfy \eqref{defl2}.
  For $n \geq 6$, let $T_n$ be the polyomino obtained by appending the polyomino of Figure~\ref{F:2d-case}(c) to the right of $T_{n-4}$.

  By induction on $n$, we have $n_1(T_n) = \ell_\squ(n)$ for all $n \geq 2$, since the fact that appending the T-shaped polyomino of Figure \ref{F:2d-case}(c) adds $4$ cells and $3$ leaves, but subtracts $1$ leaf.
\end{proof}

In order to prove that the family $\{T_n \mid n \geq 2\}$ described in the proof of Lemma~\ref{L:family-2d} is maximal, we need the following result characterizing particular subtrees that appear in possible counter-examples of minimum size.
\begin{lemma}\label{L:24}
  Let $T$ be a tree-like polyomino of minimum size such that $n_1(T) > \ell_\squ(n(T))$ and let $T'$ be a tree-like polyomino such that $n(T') = n(T) - i$, for some $i \in \{1,3,4\}$.
  Also, let $\Delta\ell_\squ(1) = 0$, $\Delta\ell_\squ(3) = 1$ and $\Delta\ell_\squ(4) = 2$.
  Then $n_1(T) > n_1(T') + \Delta\ell_\squ(i)$.
\end{lemma}

\begin{proof}
It is easy to prove by induction that for any $k \geq 2$, $\ell_\squ(k + i) \geq \ell_\squ(k) + \Delta\ell_\squ(i)$, where $i \in \{1,3,4\}$.
Therefore,
\begin{alignat*}{2}
  n_1(T) &   >   \ell_\squ(n(T)),  {}&&\quad \mbox{by assumption,}\\
         &   =   \ell_\squ(n(T') + i), {}&& \quad  \mbox{by definition of $T'$,} \\
         & \geq  \ell_\squ(n(T')) + \Delta\ell_\squ(i), {}&& \quad \mbox{by the observation above,}\\
         & \geq  L_\squ(n(T')) + \Delta\ell_\squ(i), {}&&  \quad \mbox{by minimality of $n(T)$,} \\
         & \geq  n_1(T') + \Delta\ell_\squ(i), {}&& \quad \mbox{by definition of $L_\squ$,}
\end{alignat*}
concluding the proof.
\end{proof}

We are now ready to prove that the family $\{T_n \mid n \geq 2\}$ is maximal.
\begin{lemma}\label{L:upper-2d}
  For all $n \geq 2$, $ L_\squ(n) \leq \ell_\squ(n) $.
\end{lemma}

\begin{proof}
  Suppose, by contradiction, that $T$ is a tree-like polyomino of minimal size such that $n_1(T) > \ell_\squ(n(T))$.
  We first show that all vertices of  $T$ of depth $1$ have degree $3$ or $4$.
  Arguing by contradiction, assume that there exists a vertex $u_1$ of $T$ such that $\depth_T(u_1) = 1$ and $\deg_T(u_1) = 2$.
  Let $T'$ be the tree-like polyomino obtained from $T$ by removing the leaf adjacent to $u_1$ (see Figure~\ref{F:2d-case}(e)).
  Then $n(T') = n(T) - 1$ and $n_1(T') = n_1(T)$, contradicting Lemma~\ref{L:24}.

  Now, we show that $T$ cannot have a vertex of depth $2$.
  Again by contradiction, assume that such a vertex $u_2$ exists.
  Clearly, $\deg_T(u_2) \neq 4$, otherwise $u_2$ would have a neighbor of depth $1$ and degree $2$, which was just  shown to be impossible.
  If $\deg_T(u_2) = 3$, then we are either in case (f) or (g) of Figure~\ref{F:2d-case}.
  In each case, let $T'$ be the tree-like polyomino obtained by removing the four gray cells.
  Then $n(T') = n(T) - 4$ and $n_1(T') = n_1(T) - 2$, contradicting Lemma~\ref{L:24}.
  Finally, if $\deg_T(u_2) = 2$, then either (h) or (i) of Figure~\ref{F:2d-case} holds, leading to a contradiction with Lemma~\ref{L:24} when removing the gray cells.
  Since every tree-like polyomino of size larger than $6$ has at least one vertex of depth $2$, the proof is completed by exhaustive inspection of all tree-like polyominoes of size at most $6$.
\end{proof}

Combining Lemmas~\ref{L:family-2d} and \ref{L:upper-2d}, we have proved the following result.
\begin{theorem}\label{T:2d}
  For all integers $n\geq 2$, $L_\squ(n) = \ell_\squ(n)$ and the asymptotic growth of $L_\squ$ is given by $L_\squ(n) \sim \frac{1}{2}\cdot n$.
\end{theorem}

%
%
%
%

\section{Fully Leafed Tree-Like Polyhexes and Polyiamonds}
\label{S:polyhexes-polyiamonds}

In the hexagonal and triangular lattices, the other two regular lattices of the plane, the leaf functions for tree-like polyforms are easy to compute.

We first consider the hexagonal lattice \emph{Hex} in which each cell is a regular hexagon of radius $1$.
If $p \in \R^2$ is the center of a hexagonal cell, then the center of its $6$ neighbors are
$$p + \vv{v_\theta}, \quad \theta = \frac{k\cdot\pi}3, k = 0, 1, 2, 3, 4, 5,$$
where $\vv{v_\theta}$ is the vector of norm $\sqrt{3}$ in the direction $\theta$. 
This neighborhood defines a \emph{$6$-adjacency relation} in $Hex$.
Connected sets of hexagonal cells under this relation are called \emph{polyhexes} and we respectively denote by $Hex_t(n)$ and $Hex_i(n)$ the sets of fixed and free tree-like polyhexes of size $n$.
The three lines supporting the vectors $\vv{v_\theta}$ are called the axes of $Hex$. 

We show next that  the function $\ell_{\hex}$ defined by
\begin{align}\label{recurr_hex}
  \ell_{\hex}(n) =&
  \begin{cases}
    2,                      & \mbox{if $n = 2, 3$;} \\
    \ell_{\hex}(n - 2) + 1, & \mbox{if $n \geq 4$.} \\
  \end{cases}\\ \nonumber
  =& \left\lfloor\frac{n}{2} \right\rfloor+1
\end{align}
gives the number of leaves in fixed fully leafed tree-like polyhexes.

\begin{theorem}\label{leaf_function_hex}
  For all integers $n \geq 2$, $\ell_{\hex}(n) = L_{\hex}(n)$.
\end{theorem}

\begin{proof}
  We first prove that $L_{\hex}(n) \geq \ell_{\hex}(n)$.
  We exhibit a family of polyhexes that satisfies recurrence \eqref{recurr_hex}.
  For $n$ even, Figure \ref{F:polyhex-even}$(b)$ shows a sample of a fully leafed polyhex that contains an even number of cells and satisfies \eqref{recurr_hex}.
  This polyhex can easily be modified to contain an arbitrary number $k$ of cells of degree three, no cell of degree two and $k+2$ cells of degree one for a total of $2k+2$ cells.
  For $n=2k+1$ odd, we only have to remove one leaf from the previous polyhex with $k$ cells of degree $3$ and $k+2$ leaves (see Figure \ref{F:polyhex-even}$(c)$) in order to satisfy \eqref{recurr_hex}. 

  It remains to show that $L_{\hex}(n) \leq \ell_{\hex}(n)$.
  Arguing by contradiction, assume that there exists a tree-like polyhex $T$ of minimal size $n\geq 3$ such that $n_1(T) > \ell_{\hex}(n(T))$.
  Every polyhex of size $n\geq 3$ contain at least one cell of depth one.
  Let $u$ be a vertex of $T$ of depth $1$.
  Notice that $\deg_T(u) \in \{2,3\}$.
  Assume first that $\deg_T(u) = 2$ and let $T'$ be the tree-like polyhex obtained from $T$ by removing the leaf adjacent to $u$.
  Then $n_1(T') = n_1(T) > \ell_{\hex}(n(T)) \geq \ell_{\hex}(n(T'))$, contradicting the minimality of $n(T)$.
  Finally, assume that $\deg_T(u) = 3$ and let $T'$ be the tree-like polyhex obtained from $T$ by removing the two leaves adjacent to $u$.
  Then $n_1(T') = n_1(T) - 1 > \ell_{\hex}(n(T)) - 1 = \ell_{\hex}(n(T'))$, contradicting the minimality of $n(T)$.
\end{proof}
\begin{figure}
  \centering
  \begin{tikzpicture}[scale=.7]
  \newcommand\hexintcell[2]{\hexcell{#1}{#2}{coldeg3}}
  \newcommand\hexleafcell[2]{\hexcell{#1}{#2}{coldeg1}}
   \newcommand\hexdegtwocell[2]{\hexcell{#1}{#2}{coldeg2}}
  \newcommand\hexcell[3]{\begin{scope}[scale=0.3, xshift=#1 cm * 2 * 0.86602540378 + 0.86602540378 * 2 * 0.5 * #2 cm, yshift=1.5 * #2 cm] \draw[fill=#3] (30:1cm) -- (90:1cm) -- (150:1cm) -- (210:1cm) -- (270:1cm) -- (330:1cm) -- cycle; \end{scope}}
  \newcommand\hexdots[2]{\begin{scope}[scale=0.3, xshift=#1 cm * 2 * 0.86602540378 + 0.86602540378 * 2 * 0.5 * #2 cm, yshift=1.5 * #2 cm] \draw[draw, very thin] (30:1cm) -- (90:1cm) -- (150:1cm) -- (210:1cm) -- (270:1cm) -- (330:1cm) -- cycle; \node at (0,0) {$\ldots$}; \end{scope}}
  \begin{scope}
    \hexintcell{-1}{0}
    \hexintcell{0}{0}
    \hexintcell{0}{1}
    \hexintcell{1}{-1}
    \hexleafcell{-2}{1}
    \hexleafcell{-1}{-1}
    \hexleafcell{-1}{2}
    \hexleafcell{1}{1}
    \hexleafcell{2}{-1}
    \hexleafcell{1}{-2}
    \node at (0,-2.5) {(a)};
  \end{scope}

  \begin{scope}[xshift=3.5cm]
    \hexintcell{-1}{0}
    \hexintcell{0}{0}
    \hexintcell{1}{-1}
    \hexintcell{2}{-1}
    \hexleafcell{-2}{1}
    \hexleafcell{-1}{-1}
    \hexleafcell{0}{1}
    \hexleafcell{1}{-2}
    \hexleafcell{2}{0}
     \hexintcell{3}{-2}

    \hexintcell{4}{-2}
    \hexintcell{5}{-3}
    \hexleafcell{4}{-1}
    \hexleafcell{6}{-3}
    \hexleafcell{5}{-4}
    \hexleafcell{3}{-3}
    \node at (0,-2.5) {(b) };
  \end{scope}
  
  \begin{scope}[xshift=8cm]
    \hexintcell{-1}{0}
    \hexintcell{0}{0}
    \hexintcell{1}{-1}
    \hexdegtwocell{2}{-1}
    \hexleafcell{-2}{1}
    \hexleafcell{-1}{-1}
    \hexleafcell{0}{1}
    \hexleafcell{1}{-2}
     \hexintcell{3}{-2}

    \hexintcell{4}{-2}
    \hexintcell{5}{-3}
    \hexleafcell{4}{-1}
    \hexleafcell{6}{-3}
    \hexleafcell{5}{-4}
    \hexleafcell{3}{-3}
    \node at (0,-2.5) {(c) };
  \end{scope}

\end{tikzpicture}
  \caption{Fully leafed tree-like polyhexes.}
  \label{F:polyhex-even}
\end{figure}
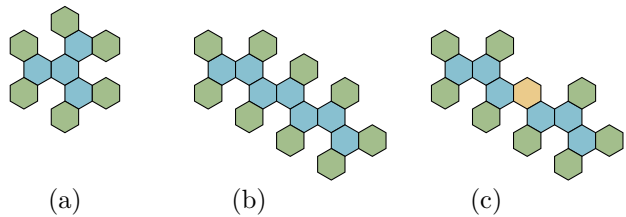

Being the dual graph of the hexagonal lattice, the triangular lattice, denoted $Triang$, presents similar properties.
Recall that the triangular lattice is the result of the tessellation of the plane with equilateral triangles.
We choose triangles of radius one with sides of length $\sqrt{3}$, one of which is horizontal, and where center to center distance between adjacent triangles is one.
If $c\in \R$ is the center of a triangular cell then the centers of its three adjacent triangles are 
$$c+\vv{v_\theta}, \quad \theta = 2k\pi / 3, k = 0, 1, 2$$
where $\vv{v_\theta}$ is the vector of length $1$ and direction $\theta$.
This defines a \emph{$3$-adjacency relation} in $Triang$ and connected sets of triangular cells under this relation are called \emph{polyiamonds}.

In the next theorem, the function $\ell_{Triang}(n)$ defined by the conditions 
\begin{align}
  \label{triang_rec}
  \ell_{Triang}(n)=&
  \begin{cases}
  2 & \mbox{if } n=2,3,\\
  \ell_{Triang}(n-2)+1 & \mbox{if } n\geq 4.
  \end{cases}\\
  \nonumber
  =&\left\lfloor\frac{n}{2}\right\rfloor+1
\end{align}
is proved to be the leaf function of fixed tree-like polyiamonds.
\begin{theorem}
For all integers $n\geq 2$, we have 
\begin{align*}
\ell_{Triang}(n)=L_{Triang}(n).
\end{align*}
\end{theorem} 
\begin{proof}
As in theorem \ref{leaf_function_hex} we first prove that $L_{Triang}(n)\geq \ell_{Triang}(n)$ by exhibiting a family of fixed polyiamonds satisfying recurrence \ref{triang_rec}.
We skip the details which are very similar to those in the proof of theorem \ref{leaf_function_hex}.
We then show by contradiction that $L_{Triang}\leq \ell_{Triang}$ with an argument identical to the one in theorem \ref{leaf_function_hex}.
\end{proof}
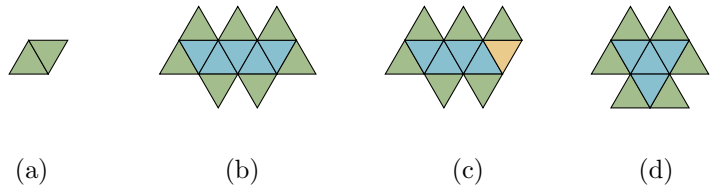
\begin{figure}
  \centering
  \begin{tikzpicture}
  \newcommand\tricell[4]{
    \begin{scope}[
      scale = 0.3, 
      xshift = #1 cm * 0.866,
      yshift = 1.5 * #2 cm - (#3 - 1) * 0.25 cm,
    ] 
      \draw[fill=#4] (-30*#3:1cm) -- (90*#3:1cm) -- (210*#3:1cm) -- cycle; 
    \end{scope}
  }
  \newcommand\triupleaf[2]{
    \tricell#1#21{coldeg1}
  }
  \newcommand\tridownleaf[2]{
    \tricell#1#2{-1}{coldeg1}
  }
  \newcommand\triupint[2]{
    \tricell#1#21{coldeg3}
  }
  \newcommand\tridowntwo[2]{
    \tricell#1#2{-1}{coldeg2}
  }
  \newcommand\tridownint[2]{
    \tricell#1#2{-1}{coldeg3}
  }

  \begin{scope}[xshift=-2cm]
    \triupleaf{1}{1}
    \tridownleaf{2}{1}
    \node () at (0.3,-1) {(a)};
  \end{scope}

  \begin{scope}
    \triupleaf{1}{1}
    \tridownint{2}{1}
    \triupleaf{2}{2}
    \triupint{3}{1}
    \tridownleaf{3}{0}
    \tridownint{4}{1}
    \triupleaf{4}{2}
    \triupint{5}{1}
    \tridownleaf{5}{0}
    \tridownint{6}{1}
    \triupleaf{6}{2}
    \triupleaf{7}{1}
    \node () at (1.1,-1) {(b)};
  \end{scope}

  \begin{scope}[xshift=3cm]
    \triupleaf{1}{1}
    \tridownint{2}{1}
    \triupleaf{2}{2}
    \triupint{3}{1}
    \tridownleaf{3}{0}
    \tridownint{4}{1}
    \triupleaf{4}{2}
    \triupint{5}{1}
    \tridownleaf{5}{0}
    \tridowntwo{6}{1}
    \triupleaf{6}{2}
    \node () at (1.1,-1) {(c)};
  \end{scope}
  \begin{scope}[xshift=6cm]
    \triupleaf{1}{0}
    \triupleaf{3}{0}
    \triupleaf{0}{1}
    \triupint{2}{1}
    \triupleaf{4}{1}
    \triupleaf{1}{2}
    \triupleaf{3}{2}
    \tridownint{1}{1}
    \tridownint{3}{1}
    \tridownint{2}{0}
    \node () at (0.6,-1) {(d)};
  \end{scope}

\end{tikzpicture}
  \caption{Fully leafed tree-like polyiamonds.}
  \label{F:polyiamonds-even}
\end{figure}
\section{Fully Leafed Tree-Like Polycubes}\label{S:polycubes}

The basic concepts introduced in Section~\ref{S:polyominoes} are now extended to tree-like polycubes with additional considerations that complexify the arguments.
Recall that for all integers $n \geq 2$, 
$$L_\cub(n) = \max\{n_1(T) \mid \mbox{$T$ is a tree-like polycube of size $n$} \}.$$
A naive tentative to extrapolate the ratio  $L_\squ(n)/n$ from polyominoes to polycubes leads to the ratio $L_\cub(n)/n=4/6$ as $n$ tends to infinity. In this section, we show that this first guess is false and that the optimal ratio is actually $28/41$ and we exhibit the geometric objects that carry this unexpected ratio.  

Define the function $\ell_\cub(n)$ as follows:

\begin{equation}\label{lcub}
\ell_\cub(n) = \begin{cases}
  f_\cub(n) + 1,        & \mbox{if $n = 6, 7, 13, 19, 25$;} \\
  f_\cub(n),            & \mbox{if $2 \leq n \leq 40$ and $n \neq 6, 7, 13, 19, 25$;} \\
  f_\cub(n-41) + 28,    & \mbox{if $41 \leq n \leq 81$;} \\
  \ell_\cub(n-41) + 28, & \mbox{if $n\geq 82$.}
\end{cases}
\end{equation}
\begin{equation}\label{fcub}
\mbox{where}\qquad\quad  f_\cub(n) = \begin{cases}
  \lfloor (2n+2) / 3 \rfloor, & \mbox{if $0 \leq n \leq 11$;}  \\ 
  \lfloor (2n+3) / 3 \rfloor, & \mbox{if $12 \leq n \leq 27$;} \\
  \lfloor (2n+4) / 3 \rfloor, & \mbox{if $28 \leq n \leq 40$.}
\end{cases}
\end{equation}

The following key observations on $\ell_\cub$ prove to be useful.
\begin{proposition}\label{P:l3}
The function $\ell_\cub$ satisfies the following properties:
\begin{enum}
  \item \label{prop:delta1}
  For all positive integers $k$, the sequence $(\ell_\cub(n+k) - \ell_\cub(n))_{n \geq 0}$ is bounded, so that the function $\Delta \ell_\cub : \N \rightarrow \N$ defined by
  $$\Delta \ell_\cub(i) = \liminf_{n\rightarrow \infty} (\ell_\cub(n+i)-\ell_\cub(n))$$
  is well-defined.
  \item  \label{prop:delta2}
  For any positive integers $n$ and $k$, if $\ell_\cub(n + k) - \ell_\cub(n) < \Delta \ell_\cub(k)$, then $n \in \{6, 7, 13, 19, 25\}$.
\end{enum}
\end{proposition}
\begin{proof}
 \ref{prop:delta1}. It is immediate that the sequence $(\ell_\cub(n+k) - \ell_\cub(n))_{n \geq 0}$ is bounded by $k$. 
 \ref{prop:delta2} is immediate from \eqref{lcub} and \eqref{fcub}. 
\end{proof}

We now introduce rooted tree-like polycubes.
\begin{definition}\label{D:rooted-polycube}
A \emph{rooted grounded tree-like polycube} is a triple $R = (T, r, \vv{u})$ such that
\begin{enum}
  \item $T = (V,E)$ is a grounded tree-like polycube of size at least $2$;
  \item $r \in V$, called the \emph{root} of $R$, is a  cell adjacent to at least one leaf of $T$;
  \item $\vv{u} \in  \Z^3$, called the \emph{direction} of $R$, is a unit vector such that $r + \vv{u}$ is a leaf of $T$.
\end{enum}
\end{definition}

When the triple $R = (T, r, \vv{u})$ is such that $r + \vv{u}$ is not a leaf of $T$, we say that $R$ is final.
The \emph{height} of $R$ is the maximum length of a path from the root $r$ to some leaf. \emph{Rooted fixed tree-like polycubes} and \emph{rooted free tree-like polycubes} are defined similarly. If $R$ is a rooted, grounded or fixed, tree-like polycube, a unit vector $\vv{v} \in  \Z^3$ is called a \emph{free direction} of $R$ whenever $r - \vv{v}$ is a leaf of $T$.
In particular, $-\vv{u}$ is a free direction of $R$.
A rooted grounded, fixed or free, tree-like polycube $R$ is called \emph{atomic} if its height is $1$.
The $11$ atomic rooted free tree-like polycubes are illustrated in Figure~\ref{arbre_basique}.

%

\begin{figure}
\begin{center}
\includegraphics[scale=.65]{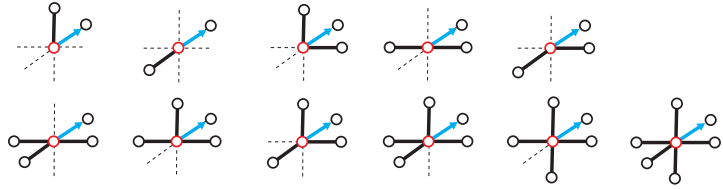}
\caption{Atomic tree-like polycubes  up to isometry} 
\label{arbre_basique}
\end{center}
\end{figure}
We now introduce an operation called the \emph{graft union} of tree-like polycubes.
\begin{definition}[Graft union]\label{D:graft}
Let $R = (T,r,\vv{u})$ and $R' = (T',r',\vv{u'})$ be rooted grounded tree-like  polycubes such that $\vv{u'}$ is a free direction of $R$. The \emph{graft union} of $R$ and $R'$, whenever it exists, is the rooted grounded tree-like  polycube
$$R \graft R' = (\Z_3[V \cup \tau(V')],r,\vv{u}),$$
where $V$, $V'$ are the sets of vertices of $T$, $T'$ respectively and $\tau$ is the translation with respect to the vector $\vv{r'r} - \vv{u'}$.
\end{definition}

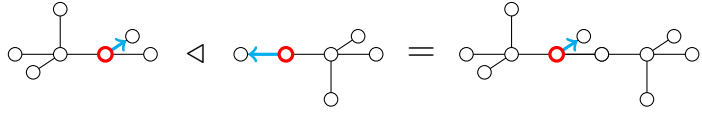
\begin{figure}[t]
\begin{center}
\begin{tikzpicture}[scale=.6, z={(-6mm,-4mm)}, sommet/.style={draw=black, fill=white, minimum size=1.8mm, circle, inner sep=0pt}, root/.style={sommet, draw=red, very thick}, fleche/.style={cyan, very thick, ->, shorten >=1mm}]
  \begin{scope}
  \draw (-2,0,0) -- (1,0,0);
  \draw (-1,0,0) -- (-1,1,0);
  \draw (-1,0,0) -- (-1,0,1);
  \draw (0,0,0)  -- (0,0,-1);
  \foreach \x/\y/\z in {0/0/0, -1/0/0, -1/1/0, -1/0/1, -2/0/0, 0/0/-1, 1/0/0} {
    \node[sommet] at (\x, \y, \z)   {};
  }
  \draw[fleche] (0,0,0) -- (0,0,-1);
  \node[root] at (0,0,0) {};
  \end{scope}

  \begin{scope}[xshift=2cm]
    \node at (0,0) {\Large$\triangleleft$};
  \end{scope}

  \begin{scope}[xshift=4cm]
  \draw (-1,0,0) -- (2,0,0);
  \draw (1,0,0) -- (1,0,-1);
  \draw (1,0,0) -- (1,-1,0);
  \foreach \x/\y/\z in {0/0/0, -1/0/0, 1/0/0, 1/-1/0, 1/0/-1, 2/0/0} {
    \node[sommet] at (\x, \y, \z)   {};
  }
  \draw[fleche] (0,0,0) -- (-1,0,0);
  \node[root] at (0,0,0) {};
  \end{scope}

  \begin{scope}[xshift=7cm]
    \node at (0,0) {\Large$=$};
  \end{scope}

  \begin{scope}[xshift=10cm]
  \draw (-2,0,0) -- (1,0,0);
  \draw (-1,0,0) -- (-1,1,0);
  \draw (-1,0,0) -- (-1,0,1);
  \draw (0,0,0)  -- (0,0,-1);
  \draw (0,0,0) -- (3,0,0);
  \draw (2,0,0) -- (2,0,-1);
  \draw (2,0,0) -- (2,-1,0);
  \foreach \x/\y/\z in {0/0/0, -1/0/0, -1/1/0, -1/0/1, -2/0/0, 0/0/-1, 1/0/0,
                        1/0/0, 0/0/0, 2/0/0, 2/-1/0, 2/0/-1, 3/0/0} {
    \node[sommet] at (\x, \y, \z)   {};
  }
  \draw[fleche] (0,0,0) -- (0,0,-1);
  \node[root] at (0,0,0) {};
  \end{scope}
\end{tikzpicture}
\caption{A well-defined, non-final graft union of two rooted grounded tree-like polycubes} 
\label{somarbre}
\end{center}
\end{figure}

The graft union is naturally extended to fixed and free tree-like  polycubes. In the latter case however, $R \graft R'$ is not a single rooted free tree-like  polycube, but rather the set of all possible graft unions obtained from an isometry.
Observe that  graft union is a partial application on rooted grounded tree-like  polycubes, i.e. the triple $(\Z_3[V \cup \tau(V')],r,\vv{u})$ is not always a rooted tree-like polycube. More precisely, the induced subgraph $\Z_3[V \cup \tau(V')]$ is always connected, but not always acyclic. Also, $r+\vv{u}$ needs not be a leaf.
Therefore, we say that a graft union $R \graft R'$ is
\begin{enum}
  \item \emph{non-final} if $R \graft R'$ is a rooted grounded tree-like  polycube;
  \item \emph{final} if the graph $G = \Z_3[V \cup \tau(V')]$ is a tree-like polycube, $\vv{u'} = -\vv{u}$ and $r+\vv{u}$ is not a leaf of $G$;
  \item \emph{well-defined} if it is either non-final or final;
  \item \emph{invalid} otherwise.
\end{enum}

Figure \ref{somarbre} illustrates a well-defined graft union of two rooted tree-like polycubes.  
The graft union interacts well with the functions $n(R)$ and $n_i(R)$ giving respectively the total number of cells and the number of cells of degree $i$ in $T$.
\begin{lemma}\label{L:graft}
Let $R_1$, $R_2$ be rooted grounded tree-like polycubes such that $R_1 \graft R_2$ is well-defined. Then
\begin{eqnarray*}
  n_1(R_1 \graft R_2) & = & n_1(R_1) + n_1(R_2) - 2, \\
  n_i(R_1 \graft R_2) & = & n_i(R_1) + n_i(R_2), \quad \mbox{for $i \geq 2$}; \\
  n(R_1 \graft R_2)   & = & n(R_1) + n(R_2) - 2.
\end{eqnarray*}
\end{lemma}

\begin{proof}
This is an immediate consequence of Definition~\ref{D:graft}.
\end{proof}

We are now ready to define a family of fully leafed tree-like polycubes.
\begin{lemma}\label{L:family-3d}
For all integer $k \geq 2$, $L_\cub(k) \geq \ell_\cub(k)$.
\end{lemma}

\begin{proof}
We exhibit a family of tree-like polycubes $\{U_k \mid k \geq 2\}$ realizing $\ell_\cub$, i.e. such that $n_1(U_k) = \ell_\cub(k)$ for all $k \geq 2$. First, for $k = 6,7,13,19,25$, let $U_k$ be the tree-like polycubes depicted in Figure~\ref{F:family-3d}(a), (b), (c), (d) and (e) respectively. It is easy to verify that $n_1(U_k) = \ell_\cub(k)$ in these cases.

\begin{figure}[tb]
  \centering
  \input{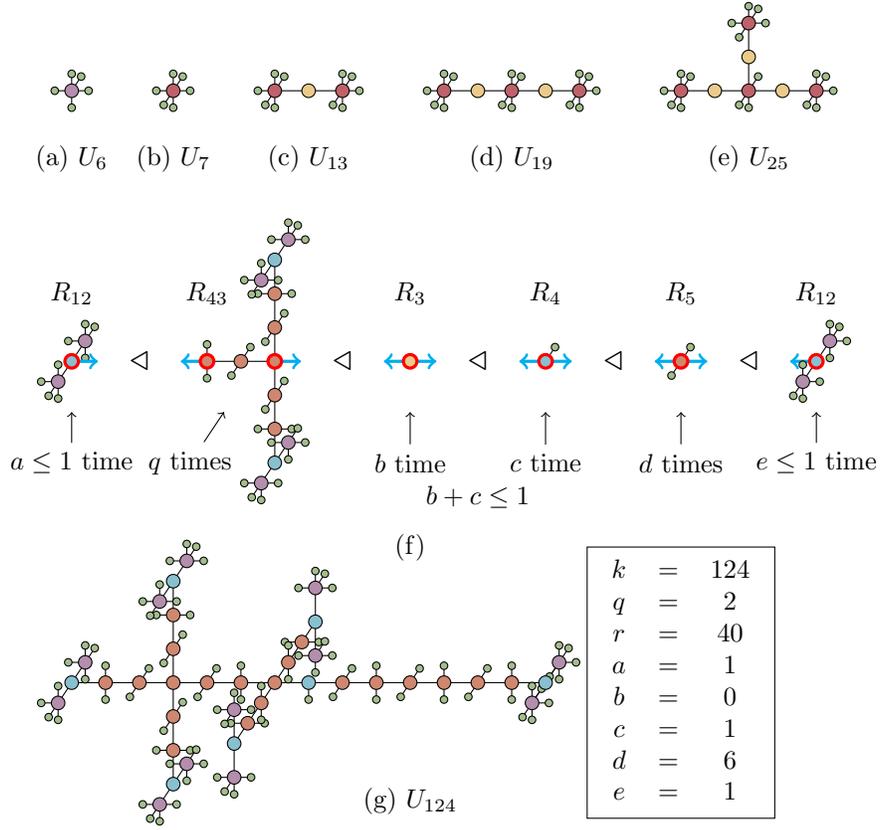}
  \caption{Fully leafed tree-like polycubes}
  \label{F:family-3d}
\end{figure}

Now, for $k \notin \{6,7,13,19,25\}$, let $q$ and $r$ be the quotient and remainder of the division of $n-2$ by $41$ and  define the  integers $a,b,c,d,e$ as follow.
\begin{alignat*}{3}
  a & = \chi(r \geq 10) \\
  b & = \chi(r \in \{1,4,7,10,11,14,17,20,23,26,27,30,33,36,39\}) \\
  c & = \chi(r \in \{2,5,8,12,15,18,21,24,28,31,34,37,40\}) \\
  d & = \left\lfloor \left(r - 10\left(\chi(r \geq 10) + \chi(r \geq 26)\right)\right) / 3 \right\rfloor \\
  e & = \chi(r \geq 26),
\end{alignat*}
where $\chi$ is the usual characteristic function. Let $U_k$ be an unrooted tree-like polycube obtained from
a rooted grounded tree-like polycube $R_k$ of the form
\begin{align} \label{eq:facto} 
R_k=R_{12}^a \graft R_{43}^q \graft R_3^b \graft R_4^c \graft R_5^d \graft R_{12}^e,
\end{align}
where, for $k \in \{3, 4, 5, 12, 43\}$, $R_k$ is depicted in Figure~\ref{F:family-3d}(f), and the exponent notation is defined by
$$R_k^\alpha = \begin{cases}
  R_2,                    & \mbox{if $k = 0$;}  \\
  R_k \graft \rho(R_k^{\alpha-1}), & \mbox{if $\alpha \geq 1$.}
\end{cases}$$
where $\rho$ is the rotation $90^{\circ}$ about the ``horizontal'' axis in Figure \ref{F:family-3d} (f) and (g). In other words, when several copies of $R_5$ or $R_{43}$ are grafted to themselves, the old graft  is rotated by $90^{\circ}$ before being grafted again. We assume that the roots and directions used for the graft union are respectively as depicted in Figure~\ref{F:family-3d} (f) by red dots and blue arrows. Note also that the two rooted grounded tree-like polycubes $R_{12}$ at each end of Figure \ref{F:family-3d} (f) are shown in the proper position up to a rotation of $90^{\circ}$.
Clearly, all graft unions in Equation \eqref{eq:facto} are well-defined and it follows from Lemma~\ref{L:graft} that
$n(R_k)   = 41q + 10(a + e) + b + 2c + 3d + 2$ and
$n_1(R_k) = 28q + 7(a + e) + c + 2d + 2 = \ell_\cub(n(R_k))$
(The recursive part in the definition of $\ell_\cub(k)$ is straightforward since $q$ is arbitrarily large and $n_1(R_{43}) = 28$.) Hence, for $k \geq 2$ and $k \notin \{6,7,13,19,25\}$, we obtain $U_k$ by taking the unrooted version of $R_k$, concluding the proof. Figure \ref{F:family-3d}(g) shows the tree-like polycube $U_{124}$ obtained from $R_{124}$ in Equation \eqref{eq:facto} with $k = 124$, and the values of $q$, $r$, $a$, $b$, $c$, $d$ and $e$ are indicated in the box to the bottom right.
\end{proof}

We now introduce a notation for the operation of graft factorization associated to the graft union of tree-like polycubes. 

\begin{definition}[Branch]\label{D:branch}
Let $T = (V,E)$ be a tree-like  polycube and $r,r'$ two adjacent vertices of $T$. Let $V_r$ and $V_{r'}$ be the set of vertices of $T$ defined by
\begin{enum}
  \item $r \in V_r$, $r' \in V_{r'}$,
  \item the subgraphs of $T$ induced by $V_r$ and $V_{r'}$ are precisely the two connected components obtained from $T$ by removing the edge $\{r,r'\}$.
\end{enum}

Then the rooted tree-like  polycube $B = (T[V_r \cup \{r'\}], r, \vv{rr'})$ is called a \emph{branch} of $T$ and the rooted tree-like  polycube $B^c = (T[V_{r'} \cup \{r\}], r', \vv{r'r})$ is called the \emph{co-branch} of $B$ in $T$. When neither $r$ nor $r'$ are leaves of $T$, then we say that $B$ and $B^c$ are \emph{proper branches} of $T$.
\end{definition}
\begin{proposition}
Let $T$ be a tree-like  polycube and $B$ a proper branch of $T$. Then both $B \graft B^c$ and $B^c \graft B$ are well-defined and final, while their corresponding unrooted tree-like polycube is precisely $T$.
\end{proposition}

\begin{proof}
This follows from Definitions~\ref{D:graft} and \ref{D:branch}.
\end{proof}

We wish to identify branches appearing in potential counter-examples, which would need to have many leaves with respect to their number of cells. 
\begin{definition}\label{D:substitution}
Let $R, R'$ be two rooted tree-like  polycubes having the same direction. We say that $R$ is \emph{substitutable} by $R'$ if, for any tree-like  polycube $T$ containing the branch  $R$, $R^c \graft R'$ is well-defined.
\end{definition}
In other words, $R$ can always be replaced by $R'$ without creating a cycle whenever $R$ appears in some tree-like polycube $T$. A sufficient condition for $R$ to be  a substitutable rooted tree-like polycube is related to its \emph{hull} (see the first paragraph of Section~\ref{S:preliminaries}).
\begin{proposition}
Let $R$ and $R'$ be two rooted tree-like polycubes with respective roots $r$ and $r'$ and such that $R' \setminus \{r'\}$ is included in $\Hull(R \setminus \{r\})$.
Then $R$ is substitutable by $R'$.
\end{proposition}

\begin{proof} Let $T$ be any rooted tree-like polycube with a branch $R$. Then we can write $T = R \graft R^c$ with $r,r'$ the respective roots of $R$ and $R^c$. 
We have to prove that $T' = R' \graft R^c$ is a tree-like polycube.
Clearly, $T'$ is connected since it is obtained by the union graft of two connected tree-like polycubes.
It remains to prove that $T'$ is acyclic.
Arguing by contradiction, assume that there is a cycle in $T'$.
Since both $R'$ and $R^c$ are acyclic, the cycle must contain some edge $\{u,v\}$ of $\treelikepolyset$ such that $u \in R' \setminus \{r,r'\}$ and $v \in R^c \setminus \{r,r'\}$.
But $u \in \Hull(R \setminus \{r\})$ and the definition of hull imply $v \in \Ext(R \setminus \{r\})$, i.e., $v$ is adjacent in $\treelikepolyset$ to some cell in $R \setminus \{r\}$.
This means, that there exists a cycle in $R \graft \R^c$, contradicting the acyclicity of $T$.
\end{proof}
We are now ready to classify rooted tree-like polycubes.
\begin{definition}\label{D:sparse-abundant}
Let $R$ be a rooted tree-like polycube. We say that $R$ is \emph{abundant} if one of the following two conditions is satisfied:
\begin{enum}
  \item $R$ contains exactly two cells,
  \item There does not exist another abundant rooted tree-like polycube $R'$, such that $R$ is substitutable by $R'$, $n(R') < n(R)$ and
\begin{equation}\label{EQ:sparse}
n_1(R) - n_1(R') \leq \Delta \ell_\cub(n(R) - n(R'))
\end{equation}
\end{enum}
Otherwise, we say that $R$ is \emph{sparse}.
\end{definition}

The following observation is immediate.
\begin{proposition}\label{P:abundant}
\item All branches of abundant rooted tree-like polycubes are also abundant.
\end{proposition}

\begin{proof}
By contradiction, assume that $R$ is an abundant tree-like polycube and that $S$ is a branch of $R$ that is not abundant. Then, by Definition~\ref{D:sparse-abundant}, $S$ has more than two cells and there must exist another abundant rooted tree-like polycube $S'$ such that $S$ is substitutable by $S'$, $n(S') < n(S)$ and $n_1(S) - n_1(S') \leq \Delta \ell_\cub(n(S) - n(S'))$.
Let $R'$ be the rooted tree-like polycube obtained from $R$ by substituting its branch $S$ by $S'$. Then, clearly, $R$ is substitutable by $R'$. Moreover, $n(R) - n(R') = n(S) - n(S')$ and $n_1(R) - n_1(R') - n_1(S) - n_1(S')$, which implies $n(R') < n(R)$ as well as $n_1(R) - n_1(R') \leq \Delta \ell_\cub(n(R) - n(R'))$, contradicting the assumption that $R$ is abundant.
\end{proof}

\begin{algorithm}[tb]
  \begin{algorithmic}[1]
    \Function{AbundantBranches}{$h$ : height} : pair of maps
      \State For $i = 1,2,\ldots,h$, let $A[i] \gets \emptyset$ and $F[i] \gets \emptyset$
      \State $A[1], F[1] \gets \{$atomic free tree-like  polycubes of size $5$ and $6\}$
      \For{$i \gets 1,2,\ldots,h$}
        \For{\textbf{each} atomic rooted free tree-like  polycube $B$}
          \For{\textbf{each} $B' \in B \graft \cup_{j=0}^{i-1} A[j]$ of height $i$}
            \If{$B'$ is abundant}
              \State \textbf{if} $B'$ is final \textbf{then} $F[i] \gets F[i] \cup B'$
              \State \textbf{else} $A[i] \gets A[i] \cup B'$
            \EndIf
          \EndFor
        \EndFor
      \EndFor
      \State \Return $(A,F)$
    \EndFunction
  \end{algorithmic}
  \caption{Generation of all abundant rooted tree-like polycubes.}\label{A:enumeration}
\end{algorithm}

Using Definition~\ref{D:sparse-abundant}, one can enumerate all abundant rooted tree-like polycubes up to a given height, 
both final and nonfinal, using a brute-force approach as described by Algorithm~\ref{A:enumeration}. In Algorithm~\ref{A:enumeration}, for a given  integer $h>0$ and each height $i = 1,2,\ldots,h$, the abundant final and nonfinal rooted tree-like polycubes are stored respectively in the two lists $F[i]$ and $A[i]$.

Algorithm~\ref{A:enumeration} was implemented in both Python \cite{bitbucket} and Haskell \cite{haskell} and run with increasing values of $h$. It turned out that there exists no abundant rooted tree-like polycube for $h = 11$, i.e. $|A[11]| = |F[11]| = 0$. Due to a lack of space, we cannot exhibit all abundant rooted tree-like polycubes, but we can give some examples. For instance, in Figure~\ref{F:family-3d}, any rooted version of the trees $U_6$, $U_7$ and $U_{12}$ is abundant, while rooted versions of $U_3$, $U_4$, $U_5$, $U_{13}$, $U_{19}$, $U_{25}$ and $U_{43}$ are sparse.


The following facts are directly observed by computation.
\begin{lemma}\label{L:observations}
Let $T$  be an abundant rooted tree-like polycube. Then
\begin{enum}
  \item  The height of $T$ is at most $10$.
  \item \label{obs:inspection}If $T$ is final  then $n_1(T) \leq \ell_\cub(n(T))$.
  \item \label{obs:6-7-13-19-25}If $T= B \graft B^c$ and $n(T) \in \{13,19,25\}$, then either $B$ or $B^c$ is sparse.
\end{enum}
\end{lemma}

\begin{proof}
Let
$A = \bigcup_{i=1}^h A(i) \quad \text{and} \quad F = \bigcup_{i=1}^h F(i),$
where $A(i)$ and $F(i)$ are respectively the sets of abundant nonfinal and final rooted tree-like polycubes computed by Algorithm~\ref{A:enumeration} with $h = 11$. In particular, we have $|A(i)|, |F(i)| > 0$ for $1 \leq i \leq 10$, but $|A(11)| = |F(11)| = 0$ (see \cite{bitbucket,haskell}).
(i) By Proposition~\ref{P:abundant}, it is immediate that if $|A(i)| = 0$ then both $|A(i+1)|=0$ and $|F(i+1)|=0$ for any $i \geq 1$, so the result follows.
(ii) By exhaustive inspection of $F$.
(iii) Assume by contradiction that both $B$ and $B^c$ are abundant. By inspecting $F$, we must have $T \in F$, but $F$ does not contain any final, abundant, rooted tree-like polycube with $13$, $19$ or $25$ vertices. 
\end{proof}

The nomenclature ``sparse'' and ``abundant'' is better understood with the following lemma.
\begin{lemma}\label{L:abundant}
Assume that there exists a tree-like  polycube $T$ of minimum size such that $n_1(T) > \ell_\cub(n(T))$. Then every branch of $T$ is abundant.
\end{lemma}

\begin{proof}
Let $B$ be any sparse branch of $T$ and $B^c$ its co-branch so that $T = B \graft B^c$  so that $B$  can be substituted by the abundant rooted tree-like polycube $B'$. Let $T' = B' \graft B^c$ and suppose first that
\begin{equation}\label{EQ:bc}
\ell_\cub(n(B \graft B^c)) - \ell_\cub(n(B' \graft B^c)) \geq \Delta\ell_\cub(n(B) - n(B')),
\end{equation}
Then Inequation~\eqref{EQ:sparse} implies
$$\Delta\ell_\cub(n(B) - n(B')) \geq n_1(B) - n_1(B'),$$
so that
\begin{eqnarray*}
  \ell_\cub(n(T)) = \ell_\cub(n(B \graft B^c))
    & \geq & n_1(B) - n_1(B') + \ell_\cub(n(B' \graft B^c)) \\
    & \geq & n_1(B) - n_1(B') + n_1(B' \graft B^c) \\
    &   =  & n_1(B \graft B^c) =  n_1(T),
\end{eqnarray*}
contradicting the hypothesis $ n_1(T)>\ell_\cub(n(T))$.
 It follows  that
\begin{equation}\label{EQ:smaller}
  \ell_\cub(n(B \graft B^c)) - \ell_\cub(n(B' \graft B^c)) < \Delta \ell_\cub(n(B) - n(B')).
\end{equation}
By Proposition~\ref{P:l3}(ii), this implies that $n(B' \graft B^c) \in \{6,7,13,19,25\}$. Since $B'$ is abundant, Lemma~\ref{L:observations}\ref{obs:6-7-13-19-25} implies that $B^c$ is sparse. Therefore, using the above argument, either  Inequations \eqref{EQ:bc} are obtained by swapping $B$ and $B^c$, leading to a contradiction, or $B^c$ can be substituted by some abundant branch $C$ such that $n(B \graft C) \in \{6,7,13,19,25\}$. Hence, $n(B), n(B^c) \leq 25$.
To conclude, observe that $B$ and $B^c$ must be fully leafed for if $B$ (or $B^c$) were not, then $B^c$ (or $B$) would be a counter-example of size $n(B^c) < n(T)$, contradicting the minimality assumption of $T$. But exhaustive inspection of sparse and fully leafed rooted tree-like polycubes of size at most $25$ yields no counterexample $T$, concluding the proof.
\end{proof}

The following fact, together with Lemma~\ref{L:family-3d},   leads to our main result. 
\begin{lemma}\label{L:upper-3d}
For all $n \geq 2$, $L_\cub(n) \leq \ell_\cub(n)$.
\end{lemma}

\begin{proof}
By contradiction, assume that there exists a tree-like  polycube $T$ of minimum size such that $n_1(T) > \ell_\cub(n(T))$. By Lemma~\ref{L:abundant}, every branch of $T$ is abundant, so that there must exist two abundant branches $B$ and $B'$ such that $T = B \graft B'$ is final. The result follows from Lemma~\ref{L:observations}\ref{obs:inspection}.
\end{proof}
Thus we have proved
\begin{theorem}\label{T:3d} For all $n\geq 2$,
$L_\cub(n) = \ell_\cub(n)$ and the asymptotic growth of $L_\cub$ is given by $L_\cub(n) \sim \frac{28}{41}n$.
\end{theorem}

\section{Saturated Tree-Like Polyforms}
\label{S:Saturated}

It was shown in the previous sections that the functions $L_\squ$, $L_{\hex}$, $L_{\tri}$ and $L_\cub$ all satisfy linear recurrences.
Let $L$ denote any of these four functions. Then it is immediate that there exists two parallel linear functions $\overline{L}$, $\underline{L}$ and a positive integer $n_0$ such that
$$\underline{L}(n) \leq L(n) \leq \overline{L}(n), \quad \mbox{for $n \geq n_0$,}$$
In all four cases, if we add the constraint that there exist infinitely many positive integers $n>0$ for which $L(n) = \overline{L}(n)$ and $L(n) = \underline{L}(n)$, then the functions $\overline{L}(n)$ and $\underline{L}(n)$ become unique. In this section, we are interested in polyforms and polycubes of size $n$ such that $L(n) \geq \overline{L}(n)$. 
\begin{definition}[Saturated tree-like polyforms and polycubes]
A fully leafed tree-like polyform or polycube $T$ is called \emph{saturated} when $n_1(T) \geq \overline{L}(n(T))$.
We denote by $SAT_i(n)$ the set of saturated tree-like polyforms of size $n$ up to an isometry of the corresponding lattice.
\end{definition}

Sets of saturated  tree-like polyforms and polycubes possess structural properties that allow their bijective reduction  to  simpler polyforms. These bijections are, to our actual knowledge,  lattice dependent and are useful in the enumeration of saturated tree-like polyforms. 
 We describe these bijections in the subsections that follow.

\subsection{Saturated Tree-Like Polyominoes}
\label{S:Saturatedpolyomin}
The two bounding linear functions of the function $L_\squ(n)$ are 
\begin{align*}
\overline{L}_\squ(n)=\frac{n+3}{2} \quad \text{and} \quad
\underline{L}_\squ(n)=\frac{n+1}{2}.
\end{align*}
and for integers $k\geq1$,  saturated tree-like polyominoes $T$ have size $n(T)=4k+1$ and $n_1(T)=2k+2$ leaves.
\begin{proposition}
\label{P:satpoly}
Let $T$ be a saturated tree-like polyomino and $u$ a vertex of depth one of $T$. Then  $deg_T(u)=4$
\end{proposition}
\begin{proof}
The proof is done by contradiction, using a similar argument to the one that is used in Lemma \ref{L:upper-2d}.
More precisely, suppose that $T$ is a saturated tree-like polyomino of minimal size $n(T)=4k+1$ and $n_1(T)=2k+2$ leaves  that contains at least one vertex $v$ of depth one such that $2\leq \deg(v)<4$. As illustrated in Figure   \ref{F:2d-case}, since $T$ is fully-leafed, it cannot contain a vertex of depth one and 
  degree $2$. If $\deg(v)=3$, then $v$ belongs to one of the three situations depicted in Figure   \ref{F:2d-case} (f), (g) and (i).
  If we remove the two  leaves adjacent to $v$ then $T$ looses two cells and one leaf which produces a tree that contradicts the leaf-maximality of the function $L_\squ$. So there is no vertex of depth one and degree $2$ or $3$ in $T$ and the proof is complete.
\end{proof}
\begin{corollary}
  Let $T$ be a saturated  tree-like polyomino of size $n(T)=4k+1$. Then $T$ is the graft union of $k$ saturated polyominoes of size $5$ shown in Figure~\ref{F:2d-case}(d), called \emph{crosses}.
\end{corollary}
\begin{proof}
  The proof is done by induction on $k$.
  This is immediate for $k = 1$.
  For $k \geq 2$, it follows from Proposition~\ref{P:satpoly} that all vertices of $T$ of depth one have degree $4$, we have $T=B\graft T'$ for some branch $B$ of size $5$ and some saturated polyomino $T'$.
  By the induction hypothesis, $T'$ is the graft union of crosses so that $T$ is the graft union of crosses as well. 
\end{proof}
\begin{theorem}[Cross operator] \label{th:cross_squ}
There exists a bijection $\phi_\squ$ from the set $\treelikepolyset(k)$
of tree-like polyominoes of size $k$ and the set $\satpolyset(4k+1)$ of saturated tree-like polyominoes of size $4k+1$:
\begin{align*}
\treelikepolyset(k) \xrightarrow{\phi_\squ}
 \satpolyset(4k+1)
\end{align*}
\end{theorem}
\begin{proof}
Start from a tree-like polyomino $T$ of size $k$ 
(see Figure \ref{Fig:crossmapsqu}$a)$). Inflation : between each pair of consecutive rows  of $T$, insert a new empty row.   Similarly, between each pair of consecutive columns  of the resulting rectangle, insert a new empty column. We obtain a disconnected set of cells $I(T)$ with nearest neighbours at distance two as shown in Figure $\ref{Fig:crossmapsqu}b)$. Cross production:  Fill  each empty unit square adjacent to a cell of $I(T)$ so that each cell of $I(T)$ has now degree $4$.  This new set of cells is connected and forms  a saturated tree-like polyomino of size $4k+1$ as shown in Figure \ref{Fig:crossmapsqu}$c)$. $\phi_\squ(T)$ is the sequence of these two transformations starting from $T$ and we call it the cross map.  It is obvious that $\phi_\squ$ in invertible i.e. starting from any saturated tree-like polyomino $S$ of size $4k+1$, we  can erase from $S$ all cells  of degree $1$ and $2$  and then remove empty columns and rows to obtain the corresponding tree-like polyomino $\phi_\squ^{-1}(S)$ of size $k$. The map $\phi_\squ$ is thus a bijection.
\end{proof}

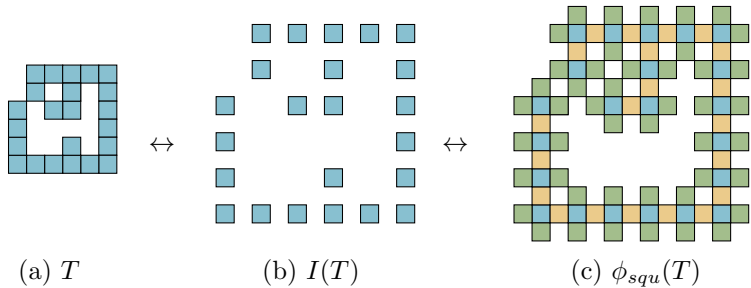
\begin{figure}
  \centering
  \begin{tikzpicture}[scale=.6]
  \newcommand\poldegthreecell[2]{\polcell{#1}{#2}{coldeg3}}
  \newcommand\polleafcell[2]{\polcell{#1}{#2}{coldeg1}}
   \newcommand\poldegtwocell[2]{\polcell{#1}{#2}{coldeg2}}
  \newcommand\polcell[3]{\begin{scope}[scale=0.4, xshift=#1 cm , yshift= #2 cm] \draw[fill=#3] (0,0) -- (1,0) -- (1,1) -- (0,1) -- (0,0)  -- cycle; \end{scope}}
 
  \begin{scope}[yshift=-0.5cm]
    \poldegthreecell{0}{0}
    \poldegthreecell{0}{1} 
    \poldegthreecell{1}{1}
    
     \poldegthreecell{2}{1} 
     
   \poldegthreecell{2}{0}
  \poldegthreecell{2}{-1}   
   \poldegthreecell{1}{-1}
        
    \poldegthreecell{3}{1} 
    \poldegthreecell{4}{1} 
     \poldegthreecell{2}{0}

    \poldegthreecell{4}{0} 
     \poldegthreecell{4}{-1} 
    \poldegthreecell{4}{-2} 
     \poldegthreecell{4}{-3} 
      \poldegthreecell{4}{-4} 
       \poldegthreecell{3}{-4} 
        \poldegthreecell{2}{-4}
         
          \poldegthreecell{2}{-3}
          
        \poldegthreecell{1}{-4} 
        \poldegthreecell{0}{-4} 
        \poldegthreecell{-1}{-4} 
        
        \poldegthreecell{-1}{-3} 
      \poldegthreecell{-1}{-2}   
      \poldegthreecell{-1}{-1}
       \node at (0.5,-3.8) {(a) $T$};
  \end{scope}

     \begin{scope}[yshift=-2cm,xshift=3cm]
  
         \node at (0,0.5cm) { $ \leftrightarrow$};
  \end{scope}


\begin{scope} [xshift=5cm]
    \poldegthreecell{0}{0}
    \poldegthreecell{0}{2} 
    \poldegthreecell{2}{2}
    \poldegthreecell{4}{2} 
    
    \poldegthreecell{4}{0}  
   \poldegthreecell{4}{-2}  
  \poldegthreecell{2}{-2}
    
   \poldegthreecell{6}{2}
 \poldegthreecell{8}{2}   
 
   \poldegthreecell{8}{0}
    \poldegthreecell{8}{-2} 
    \poldegthreecell{8}{-4} 
    \poldegthreecell{8}{-6}

    \poldegthreecell{8}{-8} 
     \poldegthreecell{6}{-8} 
    \poldegthreecell{4}{-8} 
    
     \poldegthreecell{4}{-6} 
     
     \poldegthreecell{2}{-8} 
     \poldegthreecell{0}{-8} 
     \poldegthreecell{-2}{-8}

    \poldegthreecell{-2}{-6}
    \poldegthreecell{-2}{-4} 
    \poldegthreecell{-2}{-2} 

       \node at (1.3,-4.3) {(b) $I(T)$};
  \end{scope}

 \begin{scope}[yshift=-2cm,xshift=9.5cm]
  
         \node at (0,0.5cm) { $ \leftrightarrow$};
  \end{scope}

\begin{scope} [xshift=12cm]
    \poldegthreecell{0}{0}
    \poldegthreecell{0}{2} 
    \poldegthreecell{2}{2}
    \poldegthreecell{4}{2} 
    
    \poldegthreecell{4}{0}  
   \poldegthreecell{4}{-2}  
  \poldegthreecell{2}{-2}
    
   \poldegthreecell{6}{2}
 \poldegthreecell{8}{2}   
 
  \polleafcell{0}{3} 
  \polleafcell{2}{3}
   \polleafcell{4}{3}
\polleafcell{6}{3}
 \polleafcell{8}{3} 

 \polleafcell{1}{0}  
 \polleafcell{2}{1}  
  \polleafcell{2}{-1}  

 \poldegtwocell{1}{2}
 \poldegtwocell{3}{2}
 \poldegtwocell{5}{2}
 \poldegtwocell{7}{2}
 
  \polleafcell{3}{0}
 
 \poldegtwocell{4}{1}
 \poldegtwocell{4}{-1}
 \poldegtwocell{3}{-2}
  
   \poldegthreecell{8}{0}
    \poldegthreecell{8}{-2} 
    \poldegthreecell{8}{-4} 
    \poldegthreecell{8}{-6}

 
  \polleafcell{9}{2} 
  \polleafcell{9}{0}
   \polleafcell{9}{-2}
    \polleafcell{9}{-4}
     \polleafcell{9}{-6} 
     \polleafcell{9}{-8}

    \poldegthreecell{8}{-8} 
     \poldegthreecell{6}{-8} 
    \poldegthreecell{4}{-8} 
    
     
     \poldegthreecell{2}{-8} 
     \poldegthreecell{0}{-8} 
     \poldegthreecell{-2}{-8}
 
     
  \polleafcell{8}{-9} 
  \polleafcell{6}{-9}
   \polleafcell{4}{-9}
    \polleafcell{2}{-9}
     \polleafcell{0}{-9} 
     \polleafcell{-2}{-9}

\poldegtwocell{7}{-8}
\poldegtwocell{5}{-8}
\poldegtwocell{3}{-8}
\poldegtwocell{1}{-8}
\poldegtwocell{-1}{-8}
     
\poldegtwocell{8}{-7}
\polleafcell{6}{-7}
\polleafcell{4}{-7}
\polleafcell{2}{-7}
\polleafcell{0}{-7}
\poldegtwocell{-2}{-7}

    \poldegthreecell{-2}{-6}
    \poldegthreecell{-2}{-4} 
    \poldegthreecell{-2}{-2} 
    
     
  \polleafcell{-3}{-8} 
  \polleafcell{-3}{-6}
   \polleafcell{-3}{-4}
 \polleafcell{-3}{-2}
 \polleafcell{-2}{-1}
    
     \polleafcell{-1}{0} 
     \polleafcell{-1}{2}
      \polleafcell{-1}{-2}
       \polleafcell{2}{-3}
   \polleafcell{4}{-3}
   \polleafcell{5}{-2}
   \polleafcell{5}{0}
  \polleafcell{6}{1}
  \polleafcell{7}{0}
  \polleafcell{7}{-2}
  \polleafcell{7}{-4}  
 \polleafcell{7}{-6}

\poldegtwocell{8}{1}
\poldegtwocell{8}{-1}
\poldegtwocell{8}{-3}
\poldegtwocell{8}{-5}

 \polleafcell{-1}{-6} 
  \polleafcell{-1}{-4}
   \polleafcell{-1}{-2}
   \polleafcell{0}{-1}
   \polleafcell{1}{-2}

   \poldegtwocell{-2}{-5}
 \poldegtwocell{-2}{-3}
  \poldegtwocell{0}{1}

       \node at (1.5,-4.3) {(c) $\phi_{squ}(T)$};
  \end{scope}

\end{tikzpicture}
  \caption{The cross map $\phi_\squ$ for tree-like polyominoes}
  \label{Fig:crossmapsqu}
\end{figure}


From an enumeration point of view, Theorem \ref{th:cross_squ} informs us that counting saturated polyominoes of size $4k+1$ is precisely the same as counting the number of tree-like polyominoes of size $k$. It would be interesting to obtain a similar information on fully leafed tree like polyominoes that are not saturated.
Unfortunately, we do not have a complete answer at this time.
\subsection{Saturated Tree-Like Polyhexes and Polyiamonds}
\label{S:Saturatedpolyhex}
Since the leaf functions of tree-like polyhexes and polyiamonds are equal,
the bounding linear functions   $\overline{L_\hex(n)}$  and  $\overline{L_\tri(n)}$  are also equal: 
\begin{align*}
\overline{L}_\hex(n)=\overline{L}_\tri(n)=\frac{n+2}{2} \quad \text{and} \quad
\underline{L}_\hex(n)=\underline{L}_\tri(n)=\frac{n+1}{2}.
\end{align*}
This implies that for $k\geq 1$, saturated polyhexes and polyiamonds $T$ with $k$ inner cells have even size $n(T)=2k$ and that their number of leaves is $n_1(T)=k+1$.

\begin{proposition}
\label{P:satpolyhex}
Let $T$ be a saturated tree-like polyhex or polyiamond and let $u$ be a vertex of depth one of $T$. Then  $deg_T(u)=3$
\end{proposition}
\begin{proof}
We proceed by contradiction. We know that any saturated tree-like polyhex or polyiamond $T$ have size $n(T)=2k$, for some integer $k>0$.
  Let $L(n)$ denote the leaf function of polyhexes and polyiamonds.
  Suppose that $T$ contains one vertex $v$ of depth one such that $\deg(v)=2$. Since $T$ is fully leafed, if $T'$ is the result of removing the leaf adjacent to $v$ then $n_1(T')=n_1(T)>L(2k-1)=n_1(T)-1$ which contradicts Equations \ref{recurr_hex} and \ref{triang_rec}. 
\end{proof}

\begin{corollary}
  \label{C:satpolyhex}
Let $T$ be a saturated tree-like polyhex or polyiamond. Then all inner cells of $T$ have degree $3$.
\end{corollary}
\begin{proof}
This result is immediate by induction on $n(T)$ and Proposition \ref{P:satpolyhex}.
\end{proof}

Define \emph{linear polyhex} (respectively \emph{polyiamond}) as saturated structures similar to the one presented in Figure \ref{F:polyhex-even} (b) (resp. Figure \ref{F:polyiamonds-even} (b)), i.e., where the set of inner cells of degree are placed in two staggered rows.
In this subsection we show that saturated tree-like polyhexes and polyiamonds are linear and we exhibit a bijection between  the two corresponding sets of free polyforms.  Unfortunately, we do not have yet a completely unified argument for the enumeration of the two sets. We shall thus proceed separately for the enumeration of free saturated polyhexes and polyiamonds.

In the set $\Hex_i(n)$ of free tree-like polyhexes of size $n$, there exist forbidden patterns for fully leafed tree-like polyhexes that restrict their shape.
\begin{proposition}\label{prop:hex}
Let $T$ be a fully leafed tree-like polyhex.
\begin{enum}
  \item Any connected subset  of $T$ of three cells of degree three cannot form a straight line along one of the three axes of $Hex$ (see Figure \ref{Fig:Forbdeg3}$(a)$). 
  \item Any connected subset  of $T$ of four cells  of degree three  cannot form a walk that is the result of  three consecutive steps with directions 
$$\vv{v_\theta},   \vv{v}_{\theta+\pi/3},  \vv{v}_{\theta+2\pi/3} \quad\mbox{  or  }\quad
\vv{v_\theta},   \vv{v}_{\theta-\pi/3},  \vv{v}_{\theta-2\pi/3}
$$
for any initial direction $\theta$ (see Figure \ref{Fig:Forbdeg3}$(b)$).
  \item If $T$ is saturated and $n_1(T)>10$ then $T$ is linear as shown in Figure~\ref{F:polyhex-even} $(b)$.
\end{enum}
\end{proposition}
\begin{proof}
(i) and (ii) are immediate from geometric constraints (see Figure \ref{Fig:Forbdeg3}). 

(iii) is deduced from the fact that, because of forbidden patterns,  the tree in Figure \ref{F:polyhex-even} (a) is terminal in the sense that no cell of degree three can be added to form a saturated tree. From that observation  it follows that any polyhex $T$ with  more than $4$ inner cells of degree three is linear. 

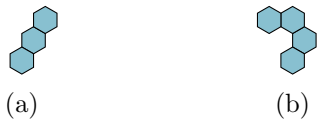
\begin{figure}
  \centering
  \begin{tikzpicture}[scale=.6]
  \newcommand\hexintcell[2]{\hexcell{#1}{#2}{coldeg3}}
  \newcommand\hexleafcell[2]{\hexcell{#1}{#2}{coldeg1}}
   \newcommand\hexdegtwocell[2]{\hexcell{#1}{#2}{coldeg2}}
  \newcommand\hexcell[3]{\begin{scope}[scale=0.3, xshift=#1 cm * 2 * 0.86602540378 + 0.86602540378 * 2 * 0.5 * #2 cm, yshift=1.5 * #2 cm] \draw[fill=#3] (30:1cm) -- (90:1cm) -- (150:1cm) -- (210:1cm) -- (270:1cm) -- (330:1cm) -- cycle; \end{scope}}
  \newcommand\hexdots[2]{\begin{scope}[scale=0.3, xshift=#1 cm * 2 * 0.86602540378 + 0.86602540378 * 2 * 0.5 * #2 cm, yshift=1.5 * #2 cm] \draw[draw, very thin] (30:1cm) -- (90:1cm) -- (150:1cm) -- (210:1cm) -- (270:1cm) -- (330:1cm) -- cycle; \node at (0,0) {$\ldots$}; \end{scope}}
  
     \begin{scope}[xshift=0 cm]
      \hexintcell{0}{1}
    \hexintcell{0}{0}
    \hexintcell{0}{2}
    \node at (0,-1) {(a)};
  \end{scope}

     \begin{scope}[xshift=6cm]
      \hexintcell{0}{1}
    \hexintcell{0}{0}
    \hexintcell{-1}{2}
    \hexintcell{-2}{2}   
    \node at (0,-1) {(b)};
  \end{scope}
  \end{tikzpicture}
  \caption{Forbidden patterns of cells in fully leafed polyhexes}
   \label{Fig:Forbdeg3}
\end{figure}

\end{proof}

Proposition \ref{prop:hex} provides a characterization of fully leafed polyhexes.
Since fully leafed polyhexes have zero or one cell of degree two, we have
\begin{align*}
n_2(T)=&0 \Leftrightarrow n(T)=2n_3(T)+2 \mbox{ and $T$ is saturated},\\
n_2(T)=&1 \Leftrightarrow n(T)=2n_3(T)+3.
\end{align*}
for any fully leafed polyhex $T$. The sets $\Hex_t(n)$ and $\Hex_i(n)$ of respectively fixed and free tree-like polyhexes of size $n$ have been well-studied in \cite{hr}.

In the set $\Tri_i$ of free polyiamonds, it is the small maximal value of the degree of a cell that imposes important restrictions on the geometry of saturated polyiamonds. 

\begin{proposition}\label{P:Polyiamondscaterpillar}
  Let $T$ be a saturated, tree-like polyiamond of size $n(T)=2k$ different from $10$. Then $T$ is a caterpillar polyiamond.
\end{proposition}
\begin{proof}
  All polyiamonds of size smaller than $10$ are obtained through a simple exploration.
 For $n(T)>10$,  We first show that $T$ must be a caterpillar by observing that the only non-caterpillar  tree-like polyiamond is shown in Figure \ref{F:polyiamonds-even} (d).
 This polyiamond cannot be extended without including a cell of degree $2$, contradicting corollary \ref{C:satpolyhex} so that all saturated tree-like polyiamonds of size $n(T)>10$ must be caterpillars. 
  It is then easy to observe that all saturated caterpillars are linear for otherwise they would contain a polyiamond isometric to the one shown in Figure \ref{F:polyiamonds-even} (d) which would forbid them from being  caterpillars.
\end{proof}
The last proposition completely defines the geometric structure of saturated polyiamonds of size different from $10$ which provides a singular counterexample.  This allows their complete enumeration which happens to be identical to that of polyhexes.

\begin{proposition}\label{prop:enumhext}
For positive integers $k$, the numbers $shex_t(2k)$ and  $stri_t(2k)$ 
of fixed saturated fully leafed polyhexes and polyiamonds of size $2k$ are  given by the single expression
\begin{align*}
shex_t(2k)=stri_t(2k)=\begin{cases}
3 & \mbox{ if } k=1,3,\\
2 & \mbox{ if } k=2,\\
8 & \mbox{ if } k=5.\\
6 & \mbox{ otherwise }
\end{cases}
\end{align*}
\end{proposition}
\begin{proof} 
  From Propositions \ref{prop:hex} $(iii)$ and \ref{P:Polyiamondscaterpillar}, we know that saturated tree-like polyhexes and polycubes of size $2k$ are caterpillars except for $k=5$.
  The cases $k=1,2,3$ are easy to verify.
  The exceptional case $k=5$ comes from the addition of the polyforms shown in Figures \ref{F:polyhex-even}$(a)$ and \ref{F:polyiamonds-even}$(d)$  which provide $2$ additional fixed polyforms 
  to the set of linear caterpillars. 
  The value $shex_t(2k)=stri_t(2k)=6$ for the remaining values of $k$ is the result of the six possible positions of fixed linear caterpillars similar to Figures \ref{F:polyhex-even}$(b)$ and \ref{F:polyiamonds-even}$(d)$. 
\end{proof}
\begin{corollary}
For positive integers $k$, the numbers $shex_i(2k)$ and $stri_i(2k)$ of free saturated tree-like polyhexes and polyiamonds of size $2k$ are given by the single expression
\begin{align*}
shex_i(2k)=stri_i(2k)=\begin{cases}
2 & \mbox{ if } k=5.\\
1 & \mbox{ otherwise }
\end{cases}
\end{align*}
\end{corollary}
\begin{proof}
This is immediate from Proposition \ref{prop:enumhext}.
\end{proof}

%

\begin{corollary}
  There exists a bijection  from free and fixed saturated tree-like polyiamonds to  polyhexes.
\end{corollary}
\begin{proof}
  The correspondance is established by simply truncating the triangles to form hexagons, as shown in Figure \ref{F:correspondance}.
  One can see that this correspondance is well defined because a saturated polytriangle is linear and its image, being linear, is also a saturated polyhex. It must also be one-to-one and onto because of the linear shape of both families of saturated polyforms. 
Moreover since  the correspondance preserves adjacency,  the image of a tree-like polyiamond is a tree-like polyhex with the same degree distribution. 
Therefore the correspondance is bijective.
\end{proof}
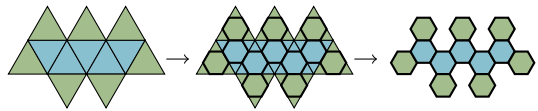
\begin{figure}
  \centering
  \begin{tikzpicture}
  \newcommand\tricell[4]{
    \begin{scope}[
      scale = 0.3, 
      xshift = #1 cm * 0.866,
      yshift = 1.5 * #2 cm - (#3 - 1) * 0.25 cm,
    ] 
      \draw[fill=#4] (-30*#3:1cm) -- (90*#3:1cm) -- (210*#3:1cm) -- cycle; 
    \end{scope}
  }
  \newcommand\triupleaf[2]{
    \tricell#1#21{coldeg1}
  }
  \newcommand\tridownleaf[2]{
    \tricell#1#2{-1}{coldeg1}
  }
  \newcommand\triupint[2]{
    \tricell#1#21{coldeg3}
  }
  \newcommand\tridownint[2]{
    \tricell#1#2{-1}{coldeg3}
  }

  \newcommand\hexcell[4]{
    \begin{scope}[
      scale = 0.3, 
      xshift = #1 cm * 0.866,
      yshift = 1.5 * #2 cm - (#3 - 1) * 0.25 cm,
    ] 
      \draw[fill=#4, thick] (0*#3:0.577cm) -- (60*#3:0.577cm) -- (120*#3:0.577cm)-- (180*#3:0.577cm)-- (240*#3:0.577cm)   -- (300*#3:0.577cm) -- cycle; 
    \end{scope}
  }
  \newcommand\hexupleaf[2]{
    \hexcell#1#21{coldeg1}
  }
  \newcommand\hexdownleaf[2]{
    \hexcell#1#2{-1}{coldeg1}
  }
  \newcommand\hexupint[2]{
    \hexcell#1#21{coldeg3}
  }
  \newcommand\hexdownint[2]{
    \hexcell#1#2{-1}{coldeg3}
  }
  \begin{scope}[xshift=0cm]
    \triupleaf{1}{1}
    \tridownint{2}{1}
    \triupleaf{2}{2}
    \triupint{3}{1}
    \tridownleaf{3}{0}
    \tridownint{4}{1}
    \triupleaf{4}{2}
    \triupint{5}{1}
    \tridownleaf{5}{0}
    \tridownint{6}{1}
    \triupleaf{6}{2}
    \triupleaf{7}{1}
  \end{scope}
  \begin{scope}[xshift=2.5cm]
    \triupleaf{1}{1}
    \tridownint{2}{1}
    \triupleaf{2}{2}
    \triupint{3}{1}
    \tridownleaf{3}{0}
    \tridownint{4}{1}
    \triupleaf{4}{2}
    \triupint{5}{1}
    \tridownleaf{5}{0}
    \tridownint{6}{1}
    \triupleaf{6}{2}
    \triupleaf{7}{1}

    \hexupleaf{1}{1}
    \hexdownint{2}{1}
    \hexupleaf{2}{2}
    \hexupint{3}{1}
    \hexdownleaf{3}{0}
    \hexdownint{4}{1}
    \hexupleaf{4}{2}
    \hexupint{5}{1}
    \hexdownleaf{5}{0}
    \hexdownint{6}{1}
    \hexupleaf{6}{2}
    \hexupleaf{7}{1}
  \end{scope}
  \begin{scope}[xshift=5cm]
    \hexupleaf{1}{1}
    \hexdownint{2}{1}
    \hexupleaf{2}{2}
    \hexupint{3}{1}
    \hexdownleaf{3}{0}
    \hexdownint{4}{1}
    \hexupleaf{4}{2}
    \hexupint{5}{1}
    \hexdownleaf{5}{0}
    \hexdownint{6}{1}
    \hexupleaf{6}{2}
    \hexupleaf{7}{1}
  \end{scope}
  \draw[->] (2.1,0.5) -- (2.4,0.5);
  \draw[->] (4.6,0.5) -- (4.9,0.5);
\end{tikzpicture}
  \caption{The bijection from  satured polyiamonds to saturated polyhexes.}
  \label{F:correspondance}
\end{figure}
\subsection{Saturated Tree-Like Polycubes}
\label{S:Saturatedpolycubes}

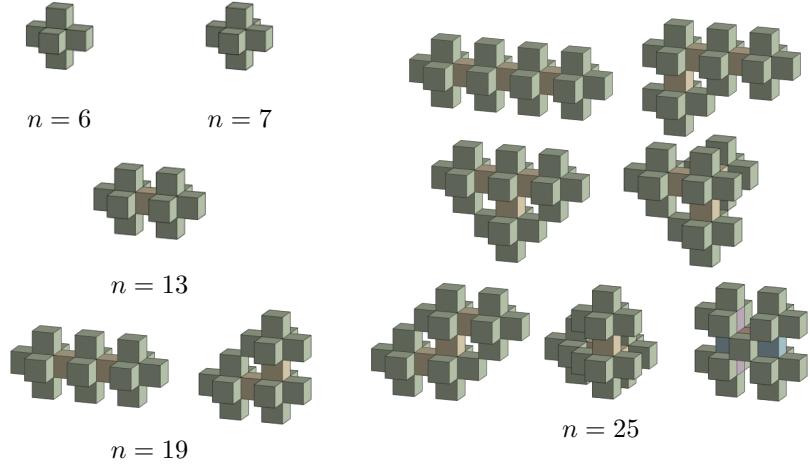
\begin{figure}[t]
  \centering
  \begin{tikzpicture}[xscale=0.8, yscale=1.1]
    \node at (   0,    0) {\includegraphics[width=9mm]{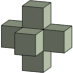}};
    \node at (   3,    0) {\includegraphics[width=9mm]{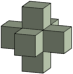}};
    \node at ( 1.5,   -2) {\includegraphics[width=15mm]{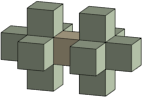}};
    \node at ( 0.5,   -4) {\includegraphics[width=21mm]{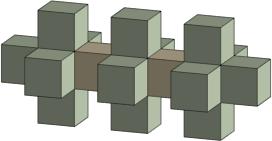}};
    \node at (3.25,   -4) {\includegraphics[width=15mm]{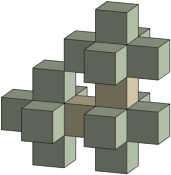}};
    \node at (   0,   -1) {$n = 6$};
    \node at (   3,   -1) {$n = 7$};
    \node at ( 1.5,   -3) {$n = 13$};
    \node at ( 1.5,   -5) {$n = 19$};

    \begin{scope}[yshift=-5mm]
    \node at ( 7.5,    0) {\includegraphics[width=27mm]{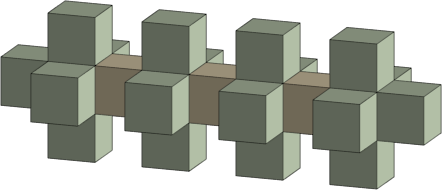}};
    \node at (  11,    0) {\includegraphics[width=21mm]{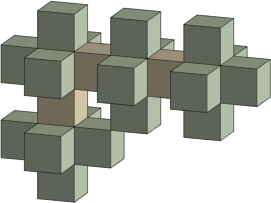}};
    \node at ( 7.5, -1.5) {\includegraphics[width=21mm]{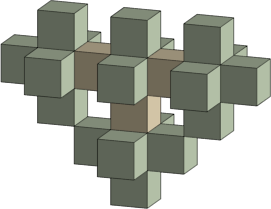}};
    \node at (10.5, -1.5) {\includegraphics[width=18mm]{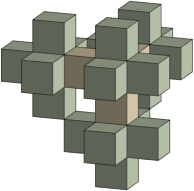}};
    \node at (11.5, -3.2) {\includegraphics[width=15mm]{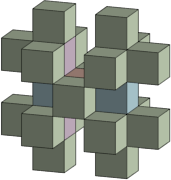}};
    \node at ( 9.0, -3.2) {\includegraphics[width=15mm]{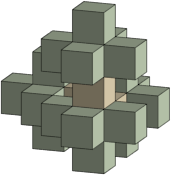}};
    \node at ( 6.5, -3.2) {\includegraphics[width=21mm]{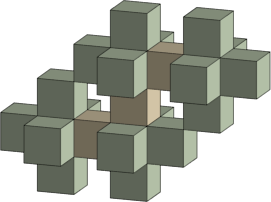}};
    \node at (   9,-4.25) {$n = 25$};
    \end{scope}
  \end{tikzpicture}
  \caption{All saturated tree-like free polycubes of size $n \in \{6,7,13,19,25\}$.}
  \label{F:saturated-basic-polycubes}
\end{figure}

In the polycube case, the leaf function is bounded by the linear functions 
$$\overline{L_\cub}(x)=\frac{28x+36}{41} \quad \text{and} \quad
\underline{L_\cub}(x)=\frac{28x-6}{41}.$$
In particular, we have $L_\cub(n) \geq \overline{L_\cub}(n)$ if and only if
$$n \in \{6,7,13,19,25\} \cup \{41k + 28 \mid k \in \N\}.$$

Let $\satcubset_{,i}(n)$
 be the set of saturated tree-like free polycubes of size $n$.
The sets $\satcubset_{,i}(n)$ for $n \in \{6,7,13,19, 25\}$ are easily found by inspection (see Figure~\ref{F:saturated-basic-polycubes}).
It is worth mentioning that, for all but one of these, the inner cells have degree  equal either to $2$ or $6$.
In fact, the rightmost, lowest tree-like polycube with $25$ cells in Figure~\ref{F:saturated-basic-polycubes} is the smallest saturated polycube with cells of degrees $3$, $4$ and $5$, a degree distribution that becomes unavoidable in larger saturated polycubes.

%

For the remainder of this subsection, we focus on saturated polycubes of size $n = 41k + 28$, for  nonnegative integers $k$.
A first key observation is that the branches occurring in such saturated tree-like polycubes are quite restricted.

\begin{figure}
  \centering
  \begin{tikzpicture}
    \node at (0,0) {\includegraphics[width=35mm]{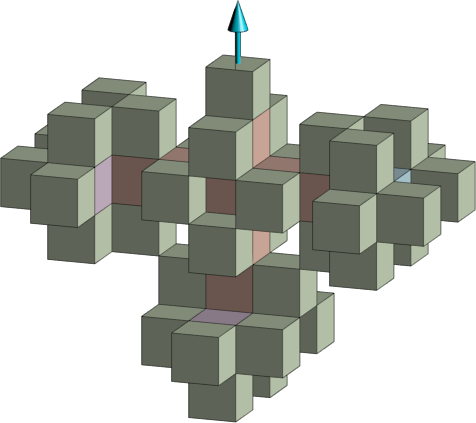}};
    \node at (4,0) {\includegraphics[width=19mm]{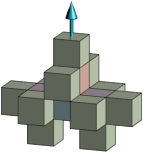}};
    \node at (8,0) {\includegraphics[width=23mm]{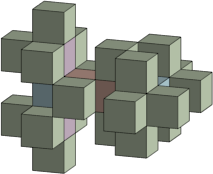}};
    \node at (0,-2) {(a)};
    \node at (4,-2) {(b)};
    \node at (8,-2) {(c)};
  \end{tikzpicture}
  \caption{The saturated branch (a) is substitutable by the saturated branch (b). (c) A saturated tree-like polycube of $28$ cells.}\label{F:superantenna}
\end{figure}
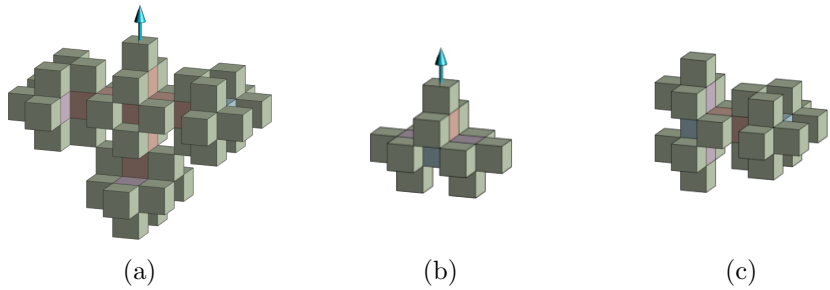

\begin{lemma}\label{L:satsurj}
Let $T \in \satcubset_i(41k + 28)$ for some integer $k \geq 0$. Then at least one of the following two conditions hold.
\begin{enum}
  \item $T$ is the tree-like polycube depicted in Figure~\ref{F:superantenna}(c);
  \item The rooted tree-like polycube $T_{56}$ depicted in Figure~\ref{F:superantenna}(a) is a branch of $T$.
\end{enum}
\end{lemma}

\begin{proof}
If $T_{56}$ is a branch of $T$, then the lemma follows immediately.
Otherwise, assume that $A_{56}$ is not a branch of $T$.
Using a slight variation of the program described in Lemma \ref{L:observations}, we observed that the only saturated branch (i.e., a branch that occurs in a saturated tree-like polycube) that can be extended is $A_{56}$.
  Therefore, by exhaustive inspection of all generated final branches, one notices that the only saturated examples are either those of Figure~\ref{F:saturated-basic-polycubes} or the one in Figure~\ref{F:superantenna}(c).
\end{proof}

A special family of tree-like polycubes is defined as follows.
 \begin{definition}[$4$-cross and $4$-tree polycubes]
A tree-like polycube of $5$ coplanar cells with exactly one inner cell of degree $4$ is called a \emph{$4$-cross}.
Moreover, for any integer $k\geq 0$, a tree-like polycube is called a \emph{$4$-tree} when it is the graft union of $k$ $4$-crosses.
 \end{definition}
A $4$-cross and a $4$-tree are depicted in Figures \ref{fig:4cross} and \ref{fig:4cross2} respectively.
It is easy to see that, for any $4$-tree $T$ and any vertex $x \in V(T)$, we have $\deg_T(x) \in \{1,4\}$.

 \begin{figure}[b]
\centering
\begin{subfigure} {0.35 \textwidth}
 \centering
\includegraphics[width=9mm]{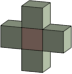}
\caption{A $4$-cross} 
\label{fig:4cross}
\end{subfigure}%
\begin{subfigure}{0.35\textwidth}
  \centering
\includegraphics[width=15mm]{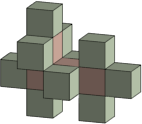}
\caption{A $4$-tree with $4$ inner cells} 
\label{fig:4cross2}
\end{subfigure}
\caption{ $4$-trees} 
\end{figure}

\begin{figure}[t]
  \centering
  \begin{tikzpicture}
  \node at (0,0) {\includegraphics[width=19mm]{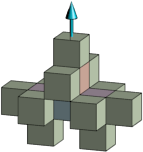}};
  \node at (0,-1.7) {$U_{15}$};
  \node at (4,0) {\includegraphics[width=24mm]{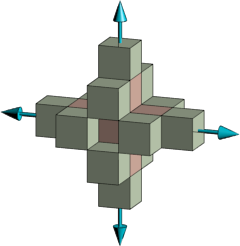}};
  \node at (4,-1.7) {$U_{17}$};
  \end{tikzpicture}
  \caption{The two basic tree-like polycubes used in the bijection $\phi$.}\label{F:molecules}
\end{figure}

 We claim that there exists a natural bijection between the set $\ftcubset(3k + 2)$ of free $4$-trees and the set $\satcubset_{,i}(41k + 28)$ of free saturated tree-like polycubes.  
 For any integer $k \geq 0$, let
 $$\phi : \ftcubset(3k + 2) \rightarrow \satcubset_{,i}(41k + 28).$$
be the function such that, for any $T \in \ftcubset(3k+2)$ with vertex set $V(T)$, edge set $E(T)$ and  $x \in V(T)$,
\begin{align}\label{eq:defphi}
\phi(x) = \begin{cases}
  U_{15}, & \mbox{if $\deg_T(x) = 1$;} \\
  U_{17}, & \mbox{if $\deg_T(x) = 4$.}
\end{cases}
\end{align}
where $U_{15}$ and $U_{17}$ are illustrated in Figure~\ref{F:molecules}
and for any edge $\{x,y\} \in E(T)$,
$$\phi(\{x,y\}) = \phi(x) \graft \phi(y).$$
so that we define $\phi(T)$ as follow:
$$\phi(T) = \mathop{\graft}\limits_{\{x,y\} \in E(T)} \phi(\{x,y\}).$$
In other words, the function $\phi$ substitutes each leaf of $T$  with a rooted, directed fully leafed tree-like polycube of $15$ cells. Each internal cell of $T$ is replaced by a $4$-tree of $17$ cells, and all these trees are grafted according to the adjacency relation in $T$.
It is worth mentioning that the result is unique up to isometry since the orientation chosen for one edge induces the orientation of all other edges in $T$.
See Figure~\ref{F:phi} for an illustration of the application of the map $\phi$ on a $4$-tree having $2$ inner cells.

\begin{figure}[t]
  \centering
  \begin{tikzpicture}
    \node at (0.5,0.0) {\includegraphics[width=12mm]{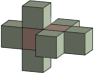}};
    \node at (2,0.5) {$\phi$};
    \node at (2,0.0) {$\longmapsto$};
    \begin{scope}[xshift=5cm]
      \node at (   0, 0.0) {\includegraphics[width=17mm]
                           {saturated/trimmed-bijection-2.png}};
      \node at (   0,-1.6) {\includegraphics[width=10mm]
                           {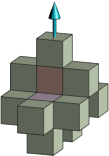}};
      \node at (   0, 1.6) {\includegraphics[width=10mm]
                           {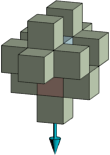}};
      \node at (-1.48, 0.14) {\includegraphics[width=13mm]
                           {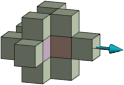}};
      \node at ( 2.4, 0.45) {\includegraphics[width=9mm]
                           {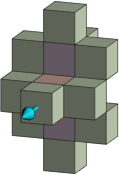}};
      \node at ( 1.7,-0.13) {\includegraphics[width=17mm]
                           {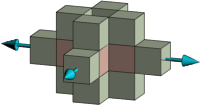}};
      \node at ( 1.0,-0.7) {\includegraphics[width=9mm]
                           {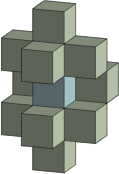}};
      \node at ( 3.20,-0.30) {\includegraphics[width=13mm]
                           {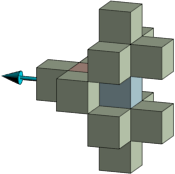}};
    \end{scope}
  \end{tikzpicture}
  \caption{The map $\phi$ applied to a $4$-tree (the arrows indicate the graft sites).}
  \label{F:phi}
\end{figure}
\begin{theorem}\label{th:saturated}
For any integer $k \geq 0$, the function $\phi$ is a bijection from $\ftcubset(3k + 2)$ to $SAT_i(41k + 28)$.
\end{theorem}

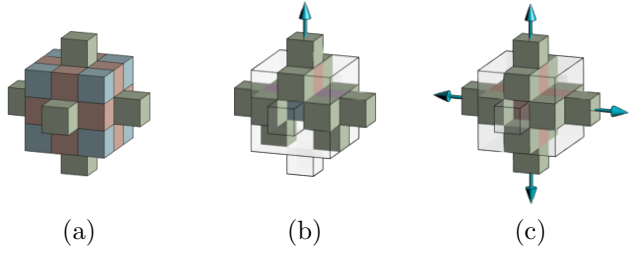
\begin{figure}
  \centering
  \begin{tikzpicture}
    \node at (0,0.2) {\includegraphics[width=19mm]{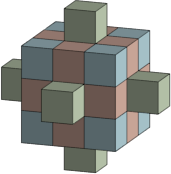}};
    \node at (3,0.4) {\includegraphics[width=19mm]{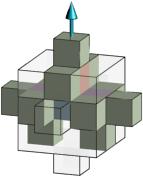}};
    \node at (6,0.2) {\includegraphics[width=26mm]{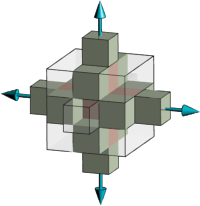}};
    \node at (0, -1.5) {(a)};
    \node at (3, -1.5) {(b)};
    \node at (6, -1.5) {(c)};
  \end{tikzpicture}
  \caption{(a) A $3 \times 3$ cube with $6$ extra cells. (b) and (c) $\phi(x)$ is included in the polycube (a).}
  \label{F:bijection-hull}
\end{figure}

\begin{proof}
Let $T\in \ftcubset(3k+2)$. 
We must verify that (i) $\phi(T)$ is acyclic, (ii) $\phi(T)$ is saturated, (iii) $\phi$ is injective and (iv) $\phi$ is surjective.

(i) Arguing by contradiction, assume that there exists some cycle in $\phi(T)$.
Since, for any edge $e$ of $T$, $\phi(e)$ is acyclic, the cycle must necessarily contain two adjacent cells in $\phi(x)$ and $\phi(y)$, for some $x,y$ non adjacent in $T$.
Clearly, the cubic cells $x$ and $y$ must share a point or a segment for  otherwise the cycle could have adjacent cells  from $\phi(x)$ and $\phi(y)$.
But both $\phi(x)$ and $\phi(y)$ are contained in translated copies $C_x$ and $C_y$ of the polycube depicted in Figure \ref{F:bijection-hull}(a).
These copies $C_x$ and $C_y$ could share either a point or a segment, but they cannot share a face, nor have intersecting cells.
Hence, no cycle cannot exist in $\phi(T)$.

(ii) Notice that $T$ has $3k+2$ cells, where $k$ of them are inner cells while $2k+2$ are leaves.
Also, $T$ contains $3k+1$ edges, which corresponds to the number of graft unions performed in the construction of $\phi(T)$.
Finally, denote respectively $L_{15}=\phi(x)$ and $I_{17}=\phi(y)$   the two images $\phi(x),\phi(y)$ of the leaves $x$ and he inner cells $y$ of $T$ shown in Equation \ref{eq:defphi}.
Then
\begin{eqnarray*}
n(\phi(T))
    & = & kn(L_{17})+(2k+2)n(L_{15})-2(3k+1)\\
    & = & 41k + 28,\\
  n_1(\phi(T))
    & = & kn_1(L_{17}) + (2k+2)n_1(L_{15}) - 2(3k + 1) \\
    & = & k \cdot 12 + (2k+2) \cdot 11 - 2(3k + 1) \\
    & = & 28k + 20,
\end{eqnarray*}
so that $\phi(T)$ is indeed saturated.

(iii) The injectivity of $\phi$ follows directly from the definition of $\phi$ and the fact that the sets $\ftcubset$ and $\satcubset_i$ are defined up to isometry.
Indeed, if $\phi(T) = \phi(T')$ for some $T,T' \in \ftcubset$, then $T$ and $T'$ must be isometric.

(iv) The proof that $\phi$ is surjective follows from Lemma \ref{L:satsurj}.
  More precisely, let $T \in \satcubset_i(41k + 28)$.
  We proceed by induction on $k$.
  \textsc{Basis}.
  If $k = 0$, then $T$ is isometric to the polycube $L_{15}\graft L_{15}$ depicted in Figure~\ref{F:superantenna}(c), so that $T = \phi(U)$, for some tree-like polycube $U$ of size $2$.
  \textsc{Induction}.
  By Lemma~\ref{L:satsurj}, the rooted tree-like polycube 
  $B=L_{15}\graft \left( L_{15}\graft L_{17}\graft L_{15}\right)$
   shown in Figure~\ref{F:superantenna}(a) is a branch of $T$, which is substitutable by the saturated branch $L_{15}$ shown in  Figure~\ref{F:superantenna}(b).
  Let $T'$ be the polycube obtained from  $T$ by the substitution of $B$ by $L_{15}$.
  By the induction hypothesis, there exists some $U' \in \ftcubset$ such that $\phi(U') = T'$.
  Let $U =  U' \graft U_5$, where $U_5$ is a $4$-cross.
   $U$ exists because $B$ is substitutable by $B'$ and the graft union $U' \graft U_5$ is well defined
   so that $\phi(U) = \phi( U') \graft \phi(U_5)$.
  Hence $T = \phi(U)$, concluding the proof.
\end{proof}

\section{Concluding Remarks}
\label{S:conclusion}

Theorems~\ref{T:2d} and \ref{T:3d} are giving the exact values  for the ratios $L_\squ(n)/n$ and $L_\cub(n)/n$ which are respectively $1/2$ and $28/41$.  For polycubes of higher dimension $d>3$, elementary arguments allow to find lower and upper bounds for $L_d(n)/n$.  Indeed, for any integer $d>2$,  it is always possible to build $d$-dimensional  tree-like polycubes by alternating cells of degree $2d$ and $2$, as was done in dimension $2$ and $3$ (see for example polycube $U_{19}$ in Figure \ref{F:family-3d}). Since this connected set of cells has the ratio 
$$L_d(n)/n=(d-1)/d,$$
 this expression becomes a lower bound for $L_d(n)/n$ in all dimensions $d > 1$.
Similarly the structure of the polycube $R_{12}$ in Figure \ref{F:family-3d} can be extended in all dimensions $d>3$ as the graft union of three cells with respective degrees $2d-1,3,2d-1$. Let $T$ be some $d$-dimensional tree-like polycube whose inner cells have degree $2d-1$, $3$ and $2d - 1$. By an inductive argument, it seems possible to prove that the ratio of any $d$-dimensional tree polycube cannot exceed the ratio $n_1(T)/n(T)=(4d-3)/4d$.

Our future work on fully leafed polyforms will include the extension of the present results to the regular lattices $\Z^d,$ for $d>3$, to nonperiodic lattices and more generally to other families of graphs.  

\bibliographystyle{plain}
\bibliography{iwoca2017-fully-leafed-polycubes}

\begin{thebibliography}{10}

\bibitem{aadhl}
J.-C. Aval, M.~D'Adderio, M.~Dukes, A.~Hicks, and Y.~Le Borgne.
\newblock Statistics on parallelogram polyominoes and a $q,t$-analogue of the
  narayana numbers.
\newblock {\em J. Comb. Th., Series A}, 123(1):271--286, 2014.

\bibitem{bfr}
E.~Barcucci, A.~Frosini, and S.~Rinaldi.
\newblock On directed-convex polyominoes in a rectangle.
\newblock {\em Discrete Mathematics}, 298(1-3):62--78, 2005.

\bibitem{bitbucket}
A.~{Blondin Mass\'e}.
\newblock A {SageMath} program to compute fully leafed tree-like polycubes,
  2016.
\newblock \url{https://bitbucket.org/ablondin/fully-leafed-tree-polycubes}.

\bibitem{haskell}
A.~{Blondin Mass\'e}.
\newblock A haskell program to compute fully leafed tree-like polycubes, 2018.
\newblock \url{https://gitlab.com/ablondin/flis-tree-polycubes}.

\bibitem{iwoca}
Alexandre Blondin~Mass\'e, Julien de~Carufel, Alain Goupil, and Maxime Samson.
\newblock Fully leafed tree-like polyominoes and polycubes.
\newblock In {\em Combinatorial Algorithms}, volume 10765 of {\em Lect. Notes
  Comput. Sci.} 28th International Workshop, IWOCA 2017, Newcastle, NSW,
  Australia, Springer, 2018.

\bibitem{bmg}
M.~Bousquet-M\'elou and A.~J. Guttmann.
\newblock Enumeration of three dimensional convex polygons.
\newblock {\em Annals of Combinatorics}, 1(1):27--53, 1997.

\bibitem{bmr}
M.~Bousquet-M\'elou and A.~Rechnitzer.
\newblock The site-perimeter of bargraphs.
\newblock {\em Adv. in Appl. Math.}, 31(1):86--112, 1997.

\bibitem{cfmrr}
G.~Castiglione, A.~Frosini, E.~Munarini, A.~Restivo, and S.~Rinaldi.
\newblock Combinatorial aspects of l-convex polyominoes.
\newblock {\em Europ. J. of Combinatorics}, 28(6):1724--1741, 2007.

\bibitem{cdcj}
J.-M. Champarnaud, J.-P. Dubernard, Q.~Cohen-Solal, and H.~Jeanne.
\newblock Enumeration of specific classes of polycubes.
\newblock {\em Electr. J. Comb.}, 20(4), 2013.

\bibitem{dv}
M.-P. Delest and G.~Viennot.
\newblock Algebraic languages and polyominoes enumeration.
\newblock {\em Theoret. Comput. Sci.}, 34:169--206, 1984.

\bibitem{gj}
M.R. Garey and D.S. Johnson.
\newblock {\em Computers and intractability}.
\newblock Freeman, 1979.

\bibitem{gcn}
A.~Goupil, H.~Cloutier, and F.~Nouboud.
\newblock Enumeration of polyominoes inscribed in a rectangle.
\newblock {\em Disc. Appl. Math.}, 158(18):2014--2023, 2010.

\bibitem{gpd}
A.~Goupil, M.~E. Pellerin, and J.~de~Wouters~d'Oplinter.
\newblock Partially directed snake polyominos.
\newblock {\em Discrete Applied Mathematics}, 236:223--234, 2018.

\bibitem{Gu}
A.~J. Guttmann.
\newblock {\em Polygons, Polyominoes and Polycubes}.
\newblock Springer, 2009.

\bibitem{hr}
F.~Harary and R.C. Read.
\newblock The enumeration of tree-like polyhexes.
\newblock {\em Proc. Edimburgh Math. So}, (17):1--13, 1970.

\bibitem{hlm}
W.~Hochst\"{a}ttler, M.~Loebl, and C.~Moll.
\newblock Generating convex polyominoes at random.
\newblock {\em Discrete Mathematics}, 153(1-3):165--176, 1996.

\bibitem{je}
I~Jensen.
\newblock Enumerations of lattice animals and trees.
\newblock {\em J. Stat. Mechanics}, 102(18):865--881, 2001.

\bibitem{ka}
W.H. Kautz.
\newblock Unit-distance error-checking codes.
\newblock {\em IRE Transactions on Electronic Computers}, 7:177--180, 1958.

\bibitem{kr}
D.A. Klarner and R.L. Rivest.
\newblock A procedure for improving the upper bound for the number of
  n-ominoes.
\newblock {\em Canadian Journal of mathematics}, 25:585--602, 1973.

\bibitem{kn}
D.E. Knuth.
\newblock Polynum, program available from knuth\textquotesingle s home-page
  at\\ \url{http://Sunburn.Stanford.EDU/~knuth/programs.html#polyominoes},
  1981.

\bibitem{bcglnv1}
A.~Blondin Mass\'e, J.~de~Carufel, A.~Goupil, M.~Lapointe, \'E. Nadeau, and
  \'E. Vandomme.
\newblock Fully leafed induced subtrees.
\newblock {\em ArXiv:1709.09808v2}, 2017.

\bibitem{bcglnv2}
A.~Blondin Mass\'e, J.~de~Carufel, A.~Goupil, M.~Lapointe, \'E. Nadeau, and
  \'E. Vandomme.
\newblock Leaf realization problem, caterpillar graphs and prefix normal words.
\newblock {\em ArXiv:1712.01942v1}, 2017.

\bibitem{Re}
D.H. Redelmeyer.
\newblock Counting polyominoes: Yet another attack.
\newblock {\em Discrete Mathematics}, 36(3):191--203, 1981.

\end{thebibliography}

\end{document}